\documentclass[10pt]{article}
\UseRawInputEncoding
\usepackage{amsmath, amssymb, amsthm, amsfonts, cases}
\usepackage{mathrsfs}
\usepackage{url}
\usepackage{authblk}
\usepackage[usenames]{color}
\usepackage{geometry}
\usepackage{indentfirst}
\allowdisplaybreaks[4]

\theoremstyle{plain}
\newtheorem{thm}{Theorem}[section]

\newtheorem{lem}[thm]{Lemma}

\theoremstyle{definition}

\newtheorem{ass}{Assumption}[section]
\newtheorem{rmk}{Remark}[section]

\makeatletter\@addtoreset{equation}{section} \makeatother

\begin{document}		
\title{ A Class of Forward-Backward Stochastic Differential Equations Driven by L\'{e}vy Processes  and Application to LQ  Problems
 \thanks{Q. Meng was supported by the  Key Projects of Natural Science Foundation of Zhejiang Province (No. Z22A013952) and the National Natural Science Foundation of China ( No.12271158 and No. 11871121). Maoning Tang was supported by the Natural Science Foundation of Zhejiang Province (No. LY21A010001).}}

\date{}

\author[a]{Maozhong Xu}
\author[b]{Maoning Tang}
\author[b]{Qingxin Meng\footnote{Corresponding author.
		\authorcr
		\indent E-mail address: xumaozhong1130@163.com (M. Xu), tmorning@zjhu.edu.cn(M. Tang), mqx@zjhu.edu.cn (Q.Meng)}}

\affil[a]{\small{Department of Mathematics, Zhejiang Normal University, Jinhua 321004, PR China}}

\affil[b]{\small{Department of Mathematical Sciences, Huzhou University, Zhejiang 313000,PR  China}}

\maketitle

\begin{abstract}
In this paper, our primary focus lies in the thorough investigation of a specific category of nonlinear fully coupled forward-backward stochastic differential equations involving time delays and advancements with the incorporation of L\'{e}vy processes, which we shall abbreviate as FBSDELDAs. Drawing inspiration from diverse examples  of linear-quadratic (LQ) optimal control problems featuring delays and L\'{e}vy processes, we proceed to employ a set of domination-monotonicity conditions tailored to this class of FBSDELDAs. Through the application of the continuation method, we achieve the pivotal results of unique solvability and the derivation of a pair of estimates for the solutions of these FBSDELDAs. These findings, in turn, carry significant implications for a range of LQ problems. Specifically, they are relevant when stochastic Hamiltonian systems perfectly align with the FBSDELDAs that fulfill the domination-monotonicity conditions. Consequently, we are able to establish explicit expressions for the unique optimal controls by utilizing the solutions of the corresponding stochastic Hamiltonian systems.
\end{abstract}


\maketitle

\section{ Introduction }
Since the seminal work of Pardoux and Peng \cite{pardoux1990adapted} on backward stochastic differential equations (BSDEs), as well as the contributions of Antonelli \cite{antonelli1993backward} regarding coupled forward-backward stochastic differential equations (FBSDEs), these equations have garnered substantial attention. They have become an essential subject of study not only due to their classical structure but also their extensive applicability across various domains, including stochastic control, finance, and economics  \cite{agram2014infinite,el1997backward,sun2016maximum}. In 1993, Antonelli \cite{antonelli1993backward} introduced coupled FBSDEs and established solvability results for a limited time interval. However, he also presented a counterexample, illustrating that the same conclusion might not hold over an extended time interval when relying solely on Lipschitz conditions. To address this challenge, many scholars have introduced additional monotonicity conditions and employed the method of continuation. This method was first introduced by Hu and Peng \cite{hu1995solution} and subsequently expanded upon by Yong \cite{peng1999fully,yong1997finding}, and others. Moreover, various other conditions and research methodologies have been proposed \cite{delarue2002existence,ma1994solving,ma2015well,pardoux1999forward,zhang2010equivalence}.

For the research of stochastic differential equations (SDEs) with L\'{e}vy processes, one of the most improtant results was given by Nualart and Schoutens \cite{nualart2000chaotic}. In their paper, by constructing a set of pairwise strongly orthonormal martingales associated with L\'{e}vy processes called Teugels martingales, they gave a martingale representation theorem related to L\'{e}vy processes. Based on this, they proved in \cite{nualart2001backward} the existence and uniqueness of solution for BSDEs with L\'{e}vy processes, and then their results were extended to the BSDEs driven by Teugels martingales and an independent multi-dimensional Brownian motion by Bahlali, Eddahbi and Essaky \cite{bahlali2003bsde}.  Subsequently, many scholars conducted more in-depth studies on BSDE driven by Teugels martingale and obtained abundant research results, seeing the reference \cite{hu2009stochastic,ren2010generalized,zhou2011solutions} and so on.  Later, there emerged a great deal of research on the stochastic control system driven by Teugels martingales and an independent Brownian motion, including the forward system, the backward system and the forward-backward system. For these results, we can refer to Meng and Tang \cite{meng2009necessary}, Tang and Zhang \cite{tang2012optimal}, Zhang et al. \cite{zhang2014stochastic} and so on.

However, in the natural and social phenomena, there exist a large number of processes whose development not only depends on their present state, but also their previous situation. Therefore, it is necessary to to study the stochastic control system with delay. In 2000, {\O}ksendal and Sulem \cite{oksendal2000maximum} obtained the stochastic maximum principle for this type of system. To our knowledge, the adjoint equation for the delayed system is also a new type of BSDE which is called anticipated BSDE. It was introduced by Peng and Yang \cite{peng2009anticipated} and they gave the proof of the unique solvability result of solution for such BSDE. Subsequently, Chen and Wu conducted a lot of research on this basis. In 2010, they studied the time-delayed SDEs in \cite{chen2010maximum}, where they obtained the maximum principle for this problem by virtue of the duality method and the anticipated BSDEs and the related application was also presented. In the following year, Chen and Wu \cite{chen2011type} continued to study a type of general FBSDEs with time-delayed SDEs as the forward equations and time-advanced BSDEs as the backward equations. Besides, \cite{chen2012dynamic} and \cite{chen2012delayed} are also their research results on the time-delayed system. For the follow-up research developments, we can refer to \cite{huang2012maximum,li2017linear,meng2015optimal,meng2021global,meng2023general,yang2022fbsdes,yu2012stochastic} and so on.

As fa as we know, the Hamiltonian systems for the stochastic control problem with L\'{e}vy processes involving time delays or advancements are all described by coupled FBSDELDAs. However, to the best of our knowledge, there is very little research on this type of FBSDEs. Therefore, in this paper, we consider the following FBSDELDA:
\begin{equation}\label{eq:1.1}
	\left\{\begin{aligned}
		dx(t)  &=b(t,\theta(t) ,\theta_{-}(t), y_{+}(t), z_{+}(t), k_{+}(t)) dt+\sigma(t, \theta(t), \theta_{-}(t), y_{+}(t), z_{+}(t), k_{+}(t)) d W(t)\\
		&\quad+\sum_{i=1}^{\infty }g^{(i)}  (t, \theta(t-), \theta_{-}(t-), y_{+}(t-), z_{+}(t), k_{+}(t))dH^{(i)}(t), \quad t \in[0, T],  \\
		dy(t)&=f(t, \theta(t), x_{-}(t), \theta_{+}(t)) dt+z(t)d W(t)+\sum_{i=1}^{\infty }k^{(i)}(t)dH^{(i)}(t), \quad t \in[0, T],\\
		x(t) &=\lambda(t),\quad y(t)=\mu(t),\quad z(t)=\rho(t),\quad k(t)=\varsigma(t),\quad t \in[-\delta , 0],\\
		y(T)&=\Phi (x(T)),\\
		x(t)&=y(t)=z(t)=k(t)=0,\quad t\in(T,T+\delta],
	\end{aligned}\right.
\end{equation}
where we denote $\theta(\cdot) =(x(\cdot)^\top ,y(\cdot)^\top ,z(\cdot)^\top ,k(\cdot)^\top )^\top$ with $k(\cdot ):=(k^{(1)}(\cdot)^\top,k^{(2)}(\cdot)^\top,\cdots )^\top$, $\theta_{-}(\cdot)=(x_{-}(\cdot)^\top ,y_{-}(\cdot)^\top ,z_{-}(\cdot)^\top ,\\k_{-}(\cdot)^\top)^\top=(x(\cdot-\delta)^\top,y(\cdot-\delta)^\top,z(\cdot-\delta)^\top,k(\cdot-\delta)^\top)^\top$, $\theta_{+}(\cdot) =(x_{+}(\cdot)^\top ,y_{+}(\cdot)^\top ,z_{+}(\cdot)^\top ,k_{+}(\cdot)^\top )^\top=(\mathbb{E} ^{\mathcal{F}_t}[x(\cdot+\delta)]^\top,\mathbb{E} ^{\mathcal{F}_t }[y(\cdot+\delta )]^\top,\mathbb{E} ^{\mathcal{F}_t }[z(\cdot+\delta )]^\top, \mathbb{E} ^{\mathcal{F}_t }[k(\cdot+\delta )]^\top)^\top$. Especially,  $\theta(\cdot-)=(x(\cdot-)^\top ,y(\cdot-)^\top ,z(\cdot)^\top ,k(\cdot)^\top )^\top$, $\theta_{-}(\cdot-)=(x_{-}(\cdot-)^\top ,y_{-}(\cdot-)^\top, \\z_{-}(\cdot)^\top ,k_{-}(\cdot)^\top )^\top=(x\big((\cdot-\delta)-\big)^\top,y\big((\cdot-\delta)-\big)^\top,z(\cdot-\delta)^\top,k(\cdot-\delta)^\top)^\top$, $y_{+}(\cdot-) =\mathbb{E} ^{\mathcal{F}_t }[y\big((\cdot+\delta )-\big)]$, where $\mathbb{E} ^{\mathcal{F}_t }[\cdot]=\mathbb{E}[\cdot|\mathcal{F}_t]$ and $\top$ denotes the transpose of matrices. Moreover, we denote $\Lambda(\cdot)=(\lambda(\cdot),\mu(\cdot),\rho(\cdot),\varsigma(\cdot))$. Let $\delta>0$ be a given constant and denote the time delay. Furthermore, we define $\mathcal{F}_t=\mathcal{F}_0$ for all $t\in [-\delta,0]$. $\left \{W_t:t\in [0 ,T ]\right \}$ is a d-dimensional standard Brownian motion. $\left \{H^{(i)}_t:t\in [0,T]\right \}_{i=1}^{\infty}$ are Teugels martingales associated with the L\'{e}vy process.
For the convenience of later use, we continue to denote
\begin{equation}\label{eq:1.2}
\Gamma (\cdot ):=(f(\cdot )^\top ,b(\cdot )^\top ,\sigma (\cdot )^\top ,g(\cdot )^\top )^\top \quad \textrm{with}\quad g(\cdot )^\top :=(g^{(1)}(\cdot )^\top,g^{(2)}(\cdot )^\top,\cdots )^\top .
\end{equation}
Then all of the coefficients of FBSDELDA \eqref{eq:1.1} are collected by $(\Lambda,\Phi ,\Gamma)$.

In 2014, Li and Wu \cite{li2014maximum}  initiated the investigation of anticipated recursive stochastic optimal control problems involving delays and L\'{e}vy processes. Their research focused on control systems described by anticipated FBSDE with delays and L\'{e}vy processes (AFBSDEDLs). In their pioneering work, they established unique solvability results for SDEs with delay and L\'{e}vy processes (SDEDLs) and anticipated BSDEs with L\'{e}vy processes (ABSDELs). These findings provided a solid foundation for similar results in the context of uncoupled AFBSDEDLs. Building upon their groundbreaking work, Li and Wu extended their research to address the Linear Quadratic (LQ) optimal control problem for systems characterized by delays and L\'{e}vy processes in a subsequent publication \cite{li2016stochastic}. This endeavor led to the solvability of stochastic Hamiltonian systems and the derivation of unique optimal control representations. It is important to note that the AFBSDEDLs they investigated in these earlier studies were uncoupled, and our research herein explores the considerably distinct fully coupled scenarios.
To prove the existence and uniqueness of solutions for fully coupled FBSDELDAs, we employ and further develop the method of continuation. Additionally, Li and Wu \cite {li2014maximum} demonstrated the continuous dependence property of solutions for ABSDELs. Therefore, our current study aims to build upon their work and delve into the continuous dependence theorem for both SDEDLs and fully coupled FBSDELDAs, as expounded in Lemma \ref{lem:2.2} and Theorem \ref{thm:3.1}.

 In 2022, in order to solve more general coupled FBSDEs which can  be applied to solve various stochastic LQ problems, Yu  \cite{yu2022forward}  has introduced various matrices, matrix-valued random variables and matrix-valued stochastic processes to present a domination-monotonicity framework and this kind of domination-monotonicity conditions are more accurate and general
 forms of the traditional monotonicity conditions which strengthen the Lipschitz condition and weaken the monotonicity
 conditions at the same time. Actually, this new framework can not only cover most of situations related to the method of continuation in the literature, but also contain many others beyond the literature. More importantly, this framework can precisely correspond to four types of LQ problems which has been demonstrated in detail in Section 4 of Yu \cite{yu2022forward}. In their paper, a unique solvability
result and a pair of estimates for coupled FBSDEs are obtained. Due to the wider applicability of this kind of framework, it has been applied to some well-known literature
\cite{li2018forward,tian2023mean,wei2021infinite,yu2022forward} and so on.

Recently,  a class of coupled FBSDEs involving time delays and time advancements on infinite horizon is studied by Yang and Yu \cite{yang2022fbsdes}, in which the unique solvability
of infinite horizon FBSDEs is obtained under a randomized Lipschitz condition and a randomized monotonicity condition. In comparison to Yang and Yu \cite{yang2022fbsdes}, we develop the domination-monotonicity conditions introduced by Yu \cite{yu2022forward} to the framework that addresses time delays and advancements associated with L\'{e}vy processes. This extension allowed us to tackle the unique solvability of FBSDELDAs and apply our approach to solve more general stochastic LQ problems that involve cost functionals with cross terms. Here it is worth mentioning that under some type of domination-monotonicity conditions,  Li, Wang and Wu \cite{li2018indefinite} studied  the uniqueness and existence of the solutions for a particular class of anticipated forward–backward stochastic differential delayed equations, where their domination-monotonicity conditions differ significantly from ours (see  Assumption \ref{ass:3.2}) and  in Remark \ref{rmk:3.1},  the corresponding difference has  been discussed in detail.

As an application of these findings, we shall re-examine stochastic LQ problems involving L\'{e}vy processes with time delays or advancements. LQ problems represent a quintessential category within the realm of stochastic optimal control problems, intensively studied by numerous scholars. When delving into the study of these LQ problems, it is imperative to engage with Hamiltonian systems, a type of linear FBSDELDAs. Leveraging the unique solvability results derived for FBSDELDAs, we obtain analogous outcomes for Hamiltonian systems in the context of LQ problems. It is noteworthy that, particularly in the case of forward LQ problems, we confront the complexity of cost functionals that include cross terms. It is also of significance to highlight that, in order to address these cross terms, we have introduced a pivotal lemma (refer to Lemma \ref{lem:4.1}) to establish the uniqueness of the optimal control.

The rest of this paper is organized as follows.  In Section 2, we introduce and establish essential notations for our analysis. Additionally, we present two key lemmas related to  SDEDLs and ABSDELs. These lemmas will prove invaluable for our subsequent analysis. In Section 3, we delve into the examination of FBSDELDA \eqref{eq:1.1}, subject to domination-monotonicity conditions. Our primary focus is on establishing the unique solvability of this equation. We also provide a pair of estimates, which are instrumental for our theoretical framework. These critical results are encapsulated in Theorem \ref{thm:3.1}. In Section 4, we build upon the findings from previous sections to address two distinct types of LQ  problems concerning systems involving Lévy processes and incorporating time delays or advancements. We successfully derive the explicit forms of unique optimal control strategies in these scenarios.

\section{Notations and Preliminaries }\label{sec:2}
Let $\mathbb{R}^n$ be the $n$-dimensional Euclidean space with the norm $|\cdot|$ and the inner product $\langle\cdot,\cdot\rangle$. Let $\mathbb{S}^n$ be the set of all symmetric matrices in $\mathbb{R}^{n\times n}$. Let $\mathbb{R}^{n\times m}$ be the collection of all $n\times m$ matrices with the norm $|A|=\sqrt{\textrm{tr}(AA^\top)}$, for $\forall A\in \mathbb{R}^{n\times m}$ and the inner product:
\begin{equation}
	\left\langle A,B\right\rangle = \textrm{tr}(AB^\top),\quad A,B\in \mathbb{R}^{n\times m}.\nonumber
\end{equation}

Let $T>0$ and $[0,T]$ denotes the finite time horizon. Let $(\Omega, \mathcal{F}, \mathbb{F}, \mathbb{P})$ be a complete filtered probability space
with a filtration $\mathbb{F}=\left\{\mathcal{F}_t : 0\le t\le T\right\}$  satisfying the usual conditions of right-continuity and $\mathbb{P}$- completeness. Besides, let the filtration $\mathbb{G}=\left\{\mathcal{G}_t : \mathcal{G}_t=\mathcal{F}_{t-\delta},0\le t\le T\right\}$.
Let $\left\{W_t:0\le t\le T\right\}$ be a d-dimensional standard Brownian motion with respect to $\mathbb{F}$ and  $\left\{S_t:0\le t\le T\right\}$ be a 1-dimensional real-valued c\`{a}dl\`{a}g trajectory called L\'{e}vy processes with stationary and independent increments, which is independent of $\left\{W_t:0\le t\le T\right\}$. It is well-known that $S_t$ has a characteristic function of the following form
\begin{equation}
E(e^{i \omega S_t})=\exp \left[i a \omega t-\frac{1}{2} \varrho^2 \omega^2 t+t \int_{\mathbb{R}}\left(e^{i \omega s}-1-i \omega \mathbf{1}_{\{|s|<1\}}\right) v(d x)\right],\nonumber
\end{equation}
where $a \in \mathbb{R}, \varrho>0$, and $v$ is a measure on $\mathbb{R}$ with $\int_{\mathbb{R}}\left(1 \wedge x^2\right) v(d x) \le \infty$. We will assume that the L\'{e}vy measure $v$ satisfies
\begin{equation}
\int_{\substack{(-\varepsilon, \varepsilon)^c}}\left(e^{k|x|}\right) v(d x) \le \infty,\nonumber
\end{equation}
for every $\varepsilon>0$ and some constant $k>0$.

We assume that
\begin{equation}
	\mathcal{F}_t=\sigma\left(S_s, s \leq t\right) \vee \sigma\left(W_s, s \leq t\right) \vee \mathcal{N},\nonumber
\end{equation}
where $\mathcal{N}$ denotes the totality of $P$-null sets.

We denote by $\left\{H^{(i)}_t:0\le t\le T\right\}_{i=1}^{\infty }$ the Teugels martingales associated with the L\'{e}vy process $\left\{S_t:0\le t\le T\right\}$. $H^{(i)}_t$ is given by
\begin{equation}\nonumber
	H^{(i)}_t=c_{i,i}Y^{(i)}_t+c_{i,i-1}Y^{(i-1)}_t+\cdots+c_{i,1}Y^{(1)}_t,
\end{equation}
where $Y^{(i)}_t=S^{(i)}_t-\mathbb{E}[S^{(i)}_t]$ for all $i\ge1$, $S^{(i)}_t$ are so-called power jump processes with $S^{(1)}_t=S_t$, $S^{(i)}_t=\displaystyle\sum_{0\le s\le t}(\Delta S_s)^{i}$ for $i\ge2$ and the coefficients $c_{i,j}$ correspond to the orthonormalization of polynomials $1,x,x^2,\cdots$ with respect to the measure $\mu(dx)=x^2v(dx)+\sigma^2\delta_0(dx)$.
Furthermore,
it is well-known that the Teugels martingales $\left\{H^{(i)}_t\right\}_{i=1}^{\infty }$ are pairwise strongly orthogonal and their predictable quadratic variation
processes are given by
\begin{equation}
	\left \langle H^{(i)}_t, H^{(j)}_t\right \rangle =\delta _{ij}t,\nonumber
\end{equation}  where $$\delta _{ij}= \left\{\begin{aligned}
		&1 \quad i=j    \\
		&0 \quad i\neq j.
	\end{aligned}\right.$$
The reader can refer to \cite {nualart2000chaotic, nualart2001backward}  for more details about Teugels martingales.

Let $\mathbb H$ be a Hilbert space with norm $\|\cdot\|_\mathbb H$, then we introduce some notations as follows:

$\bullet$ $l^{2}$: the space of all real-valued sequences $x=(x_n)_{n\ge 1}$ satisfying
\begin{equation}
	\|x\|_{l^2}:=\Big (\sum_{i=1}^{\infty } x_i^2\Big )^{1/2}<\infty .\nonumber
\end{equation}

$\bullet$ $l^{2}(\mathbb H)$: the space of all $\mathbb H$-valued sequences $f=\left\{f^i\right\}_{i\ge 1}$ satisfying
\begin{equation}
	\|f\|_{l^2(\mathbb H)}:=\Big (\sum_{i=1}^{\infty } \|f^{i}\|_\mathbb H^2\Big )^{1/2}<\infty .\nonumber
\end{equation}

$\bullet$ $C(s,r;\mathbb H)$: the space of continuous functions form $[s,r]$ into $H$.

$\bullet$ $L^2(s,r;\mathbb H)$: the space of all $\mathbb{H}$-valued  Lebesgue measurable functions $\xi(\cdot)$ satisfying
\begin{equation}
	\|\xi(\cdot)\|_{L^2(s,r;\mathbb H)}:=\bigg[\int_{s}^{r}|\xi(t)|_{\mathbb{H}}^2dt\bigg]^{1/2}<\infty .\nonumber
\end{equation}

$\bullet$ $L^{2} _{\mathcal{F}_{T} }(\Omega ; \mathbb H)$: the space of all $\mathbb H$-valued and $\mathcal{F}_{T}$-measurable random variables $\xi$ satisfying

\begin{equation}
	\|\xi\|_{L^{2} _{\mathcal{F}_{T} }(\Omega ;\mathbb{H})} :=\Big[\mathbb{E}\|\xi \|_\mathbb H^{2} \Big]^{1/2}<\infty .\nonumber
\end{equation}

$\bullet$ $L^{\infty} _{\mathcal{F}_{T} }(\Omega ;\mathbb H)$: the space of all $\mathbb H$-valued
and $\mathcal{F}_{T}$-measurable essentially bounded variables.

$\bullet$ $L^{2} _{\mathbb{F}}(s,r;\mathbb H)$: the space of all $\mathbb{H}$-valued and $\mathbb{F}$-predictable processes $f(\cdot)$ satisfying
\begin{equation}
	\|f(\cdot)\|_{L^{2} _{\mathbb{F}}(0,T;\mathbb H)} :=\bigg [\mathbb{E}\bigg (\int_{s}^{r}\|f(t)\|_{\mathbb{H}}^{2}ds\bigg ) \bigg]^{1/2}<\infty.\nonumber
\end{equation}

$\bullet$ $M^{2} _{\mathbb{F}}(s,r;\mathbb H)$: the space of all $\mathbb{H}$-valued and $\mathbb{F}$-adapted processes $f(\cdot)$ satisfying
\begin{equation}
	\|f(\cdot)\|_{L^{2} _{\mathbb{F}}(0,T;\mathbb H)} :=\bigg [\mathbb{E}\bigg (\int_{s}^{r}\|f(t)\|_{\mathbb{H}}^{2}ds\bigg ) \bigg]^{1/2}<\infty.\nonumber
\end{equation}

$\bullet$ $L^{2} _{\mathbb{F}}(s,r;l^2(\mathbb H))$: the space of all $l^2(\mathbb H)$-valued and $\mathbb{F}$-predictable
processes $f(\cdot)=\left\{f^i(\cdot)\right\}_{i\ge 1}$ satisfying
\begin{equation}
	\|f(\cdot)\|_{L^{2} _{\mathbb{F}}(0,T;l^2(\mathbb H))}
 :=\bigg [\mathbb{E}\bigg (\int_{s}^{r}\sum_{i=1}^{\infty}\|f^i(t)\|_\mathbb H^{2}ds\bigg ) \bigg]^{1/2}<\infty.\nonumber
\end{equation}

$\bullet$ $L^{\infty} _{\mathbb{F}}(s,r;\mathbb H)$: the space of all $\mathbb H$-valued and $\mathbb{F}$-predictable essentially bounded processes.

$\bullet$ $L^{\infty} _{\mathbb{G}}(s,r;\mathbb H)$: the space of all $\mathbb H$-valued and $\mathbb{G}$-predictable essentially bounded processes.

$\bullet$ $\mathcal{S}^{2} _{\mathbb{F}}(s,r;\mathbb H)$: the space of all $\mathbb H$-valued
and $\mathbb{F}$-adapted c\`{a}dl\`{a}g processes $f(\cdot)$ satisfying
\begin{equation}
	\|f(\cdot )\|_{\mathcal{S}^{2} _{\mathbb{F}}(s,r;\mathbb H)} :=\bigg[\mathbb{E} \bigg(\sup _{t\in [s,r]}\|f(t)\|_{\mathbb{H}}^{2}\bigg )\bigg ]^{1/2}<\infty. \nonumber
\end{equation}

For the sake of simplicity of notation, we will also present some product space as follows:

$\bullet$ $N_{\mathbb{F} }^{2}(0 ,T;\mathbb{R}^{n(2+d)}\times l^2(\mathbb{R}^{m}) ):= \mathcal{S}^{2} _{\mathbb{F}}(0,T;\mathbb{R}^{n})\times \mathcal{S}^{2} _{\mathbb{F}}(0,T;\mathbb{R}^{n})\times M^{2}_\mathbb{F}(0,T;\mathbb{R}^{n\times d} )\times L^{2}_\mathbb{F}(0,T;l^2(\mathbb{R}^{n}) )$. For any $\theta (\cdot )=(x(\cdot )^\top,y(\cdot )^\top,z(\cdot )^\top,k(\cdot )^\top )^\top \in N_{\mathbb{F} }^{2}(0,T ;\mathbb{R}^{n(2+d)}\times l^2(\mathbb{R}^{n}) )$, its norm is given by
\begin{equation}
	\begin{aligned}
\|\theta(\cdot ) \|_{N_{\mathbb{F} }^{2}(0,T ;\mathbb{R}^{n(2+d)}\times l^2(\mathbb{R}^{n}) )}:= \left \{ \mathbb{E}\bigg [\displaystyle\sup_{t\in [0,T]}|x(t)|^{2}+\displaystyle\sup_{t\in [0,T]}|y(t)|^{2}+\int_{0}^{T}|z(t)|^{2
}dt+\int_{0}^{T}\|k(t)\|_{l^2(\mathbb{R}^{n})}^{2}dt \bigg ] \right \}^{1/2}.\nonumber
	\end{aligned}
\end{equation}

$\bullet$ $\mathcal{N}_{\mathbb{F} }^{2}(0 ,T ;\mathbb{R}^{n(2+d)}\times l^2(\mathbb{R}^{n}) ):= L^{2} _{\mathbb{F}}(0 ,T;\mathbb{R}^{n})\times L^{2} _{\mathbb{F}}( 0 ,T;\mathbb{R}^{n})\times M^{2}_\mathbb{F}(0 ,T;\mathbb{R}^{n\times d} )\times L^{2}_\mathbb{F}(0 ,T;l^2(\mathbb{R}^{n}) )$. For any $\rho (\cdot )=(\varphi (\cdot )^\top ,\psi (\cdot )^\top,\gamma  (\cdot )^\top,\beta  (\cdot )^\top)^\top \in \mathcal{N}_{\mathbb{F} }^{2}(0 ,T ;\mathbb{R}^{n(2+d)}\times l^2(\mathbb{R}^{n}) )$, its norm is given by
\begin{equation}
\begin{aligned}
\|\rho (\cdot )\|_{\mathcal{N}_{\mathbb{F} }^{2}(0 ,T ;\mathbb{R}^{n(2+d)}\times l^2(\mathbb{R}^{n}) )}:=\left \{ \mathbb{E}\bigg [\int_{0}^{T}|\varphi(t)|^{2}dt+\int_{0}^{T}|\psi(t)|^{2}dt+\int_{0}^{T}|\gamma (t)|^{2}dt+\int_{0}^{T}\|\beta (t)\|_{l^2(\mathbb{R}^{n})}^{2}dt  \bigg ]   \right \}^{1/2}.
\end{aligned}\nonumber
\end{equation}

$\bullet$ $\mathcal{Q}(-\delta,0;\mathbb{R}^{n(2+d)}\times l^2(\mathbb{R}^{n})):=C(-\delta,0;\mathbb{R}^n)\times C(-\delta,0;\mathbb{R}^n)\times L^2(-\delta,0;\mathbb{R}^{n\times d})\times L^2(-\delta,0;l^2(\mathbb{R}^n))$. For any $\Lambda(\cdot)=(\lambda(\cdot),\mu(\cdot),\rho(\cdot),\varsigma(\cdot)) \in \mathcal{Q}(-\delta,0;\mathbb{R}^{n(2+d)}\times l^2(\mathbb{R}^{n}))$, its norm is given by
\begin{equation}
	\begin{aligned}
		\|\Lambda(\cdot)\|_{\mathcal{Q}(-\delta,0;\mathbb{R}^{n(2+d)}\times l^2(\mathbb{R}^{n}))}:=\left\{\displaystyle\sup_{t\in[-\delta,0]}|\lambda(t)|^{2}+\displaystyle\sup_{t\in[-\delta,0]}|\mu(t)|^{2}+\int_{-\delta}^{0}|\rho(t)|^{2}dt+\int_{-\delta}^{0}\|\varsigma(t)\|_{l^2(\mathbb{R}^n)}^{2}dt\right\}^{1/2}.\nonumber
	\end{aligned}
\end{equation}

$\bullet$ $\mathcal{H}[-\delta,T]:=\mathcal{Q}(-\delta,0;\mathbb{R}^{n(2+d)}\times l^2(\mathbb{R}^{n}))\times L_{\mathcal{F}_T }^{2}(\Omega ;\mathbb{R}^n)\times \mathcal{N}_{\mathbb{F} }^{2}(0 ,T ;\mathbb{R}^{n(2+d)}\times l^2(\mathbb{R}^n) )$. For any $(\pi(\cdot) ,\eta ,\rho (\cdot ))\in \mathcal{H}[-\delta ,T]$, its norm is given by
\begin{equation}
	\begin{aligned}
\|(\pi(\cdot) ,\eta ,\rho (\cdot ))\|_{\mathcal{H} [-\delta,T]}:=\left \{ \|\pi(t)\|^{2}_{\mathcal{Q}(-\delta,0;\mathbb{R}^{n(2+d)}\times l^2(\mathbb{R}^{n}))}+\|\eta \|^{2}_{L^{2} _{\mathcal{F}_{T} }(\Omega ;\mathbb{R}^n )}+\|\rho(\cdot) \|^{2}_{\mathcal{N}_{\mathbb{F} }^{2}(0 ,T ;\mathbb{R}^{n(2+d)}\times l^2(\mathbb{R}^n) )} \right \}^{1/2}.\nonumber
	\end{aligned}
\end{equation}

In what follows, we shall present some basic results on SDEDL and ABSDEL.

Firstly, we study the following SDEDL:
\begin{equation}\label{eq:2.1}
	\left\{\begin{aligned}
		&dx_{t}  =b(t, x_{t} ,x'_{t}) dt+\sigma(t, x_{t} , x'_{t}) d W_{t} +\sum_{i=1}^{\infty }g^{(i)}  (t, x_{t-} ,x'_{t-})dH_{t}^{(i)},\quad t\in [0,T],    \\
		&x_{t} =\lambda_{t},\quad t \in[-\delta , 0],
	\end{aligned}\right.
\end{equation}
where $x'_{t}=x_{t-\delta}$ and $x'_{t-}=x_{(t-\delta)-}$ .

The coefficients $(b,\sigma,g,\lambda)$ are assumed to satisfy the following conditions:
\begin{ass}\label{ass:2.1}
	$\lambda(\cdot)\in C(-\delta ,0;\mathbb{R}^n)$ and $(b,\sigma,g)$ are three  given random mappings
	\begin{equation}
		\begin{aligned}
			&b:[0,T]\times \Omega \times \mathbb{R}^n \times \mathbb{R}^n \to \mathbb{R}^n ,\\
			&\sigma :[0,T]\times \Omega \times \mathbb{R}^n \times \mathbb{R}^n \to \mathbb{R}^{n\times d} ,\\
			&g= (g^{(i)}) _{i=1}^{\infty }:[0,T]\times \Omega \times \mathbb{R}^n \times \mathbb{R}^n \to l^2(\mathbb{R}^n )
		\end{aligned}\nonumber
	\end{equation}
satisfying

	(i)For any $x, x'\in \mathbb{R}^n $, $b(\cdot ,x,x')$, $\sigma(\cdot ,x,x')$ and $g(\cdot ,x,x')$ are $\mathbb{F}$-progressively measurable. Moreover, $b(\cdot ,0,0)\in L^{2} _{\mathbb{F}}(0,T;\mathbb{R}^n )$, $\sigma (\cdot ,0,0)\in L^{2} _{\mathbb{F}}(0,T;\mathbb{R}^{n\times d} )$, $g(\cdot ,0,0)\in L^{2} _{\mathbb{F}}(0,T;l^2(\mathbb{R}^n) )$.
	
	(ii)The mappings $b$, $\sigma$ and $g$ are uniformly Lipschitz continuous with respect to $(x,x')$, i.e., for any $x,\bar{x},x',\bar{x}'\in \mathbb{R}^n$, there exists a constant $L>0$ such that
	\begin{equation}
		|b(t,x,x')-b(t,\bar{x},\bar{x}')|+|\sigma (t,x,x')-\sigma (t,\bar{x},\bar{x}')|+\|g(t,x,x')-g(t,\bar{x},\bar{x}')\|_{l^2(\mathbb{R}^n)}\le L(|x-\bar{x} |+|x'-\bar{x}' |).\nonumber
	\end{equation}
\end{ass}
\begin{lem}\label{lem:2.2}
	Under Assumption \ref{ass:2.1}, SDEDL \eqref{eq:2.1} with coefficients $( b, \sigma, g,\lambda)$ admits a unique solution $x(\cdot )\in \mathcal{S}^{2 } _{\mathbb{F}}(0,T;\mathbb{R}^n )$. Moreover, we have the following estimate:
	\begin{eqnarray}\label{eq:2.2}
		\begin{aligned}
			\mathbb E\bigg[\displaystyle\sup_{t\in[0,T]}|x_{t}|^2\bigg]
			\leqslant K\mathbb E\bigg[\displaystyle\sup_{t\in[-\delta,0]}|\lambda_{t}|^{2}+\displaystyle\int_0^T|b(t,0,0)|^2dt
			+\displaystyle\int_0^T|\sigma (t,0,0)|^2dt
			+\int_0^T\|g(t,0,0)\|_{l^2(\mathbb{R}^n)}^2dt\bigg],
		\end{aligned}
	\end{eqnarray}
where $K$ is a positive constant depending only on $T$ and the Lipschitz constant $L$. Furthermore, let $(\bar{b}, \bar{\sigma}, \bar{g},\bar{\lambda})$ be another set of coefficients satisfying Assumption \ref{ass:2.1}, and assume that $\bar{x}(\cdot )\in \mathcal{S}^{2 } _{\mathbb{F}}(0,T;\mathbb{R}^n )$ is a solution to SDEDL \eqref{eq:2.1} corresponding  the coefficients $( \bar{b}, \bar{\sigma}, \bar{g},\bar{\lambda})$.
Then the following estimate holds:
\begin{eqnarray}\label{eq:2.3}
	\begin{aligned}
		\mathbb E\bigg[\displaystyle\sup_{t\in[0,T]}|x_{t}-\bar{x}_{t}|^2\bigg]
		\le&K\mathbb E\bigg[\displaystyle\sup_{t\in[-\delta,0]}|\lambda _{t}-\bar{\lambda} _{t}|^{2}
		+\displaystyle\int_0^T|b(t,\bar{x}_{t},\bar{x}'_{t })-\bar{b}(t,\bar{x}_{t},\bar{x}'_{t})|^2dt
		\\&+\displaystyle\int_0^T|\sigma (t,\bar{x}_{t},\bar{x}'_{t})-\bar{\sigma }(t,\bar{x}_{t},\bar{x}'_{t})|^2dt
		+\int_0^T\|g(t,\bar{x}_{t},\bar{x}'_{t })-\bar{g}(t,\bar{x}_{t},\bar{x}'_{t})\|_{l^2(\mathbb{R}^n)}^2dt\bigg],
	\end{aligned}
\end{eqnarray}
where $K$ is also a positive constant which only depends only on $T$ and the Lipschitz constant $L$.
\end{lem}
\begin{proof}
	Firstly, the existence of uniqueness of the solution to SDEDL \eqref{eq:2.1} has been proved in Theorem 3.1 of Li and Wu \cite{li2014maximum}, so we just need to prove the estimate \eqref{eq:2.2} and \eqref{eq:2.3}. In the following proof, the constant $K$ may be changed line to line.
	
For simplicity, we denote by
\begin{equation}
	\left\{\begin{aligned}
		&\widehat{x}_s=x_s-\bar{x}_s,\quad \widehat{x}_t=x_t-\bar{x}_t,\\
		&\widehat{x}'_{s}=x'_{s} -\bar{x}'_{s} , \quad \widehat{\lambda}_s=\lambda_s-\bar{\lambda}_s,\\
		&\widehat{b}_s=b(s,x_s,x'_{s})-\bar{b}(s,\bar{x}_s,\bar{x}'_{s} ),\\
		&\widehat{\sigma }_s=\sigma (s,x_s,x'_{s})-\bar{\sigma }(s,\bar{x}_s,\bar{x}'_{s} ),\\
		&\widehat{g}_s=g(s,x_s,x'_{s})-\bar{g}(s,\bar{x}_s,\bar{x}'_{s} ).
	\end{aligned}\right.\nonumber
\end{equation}
Using It\^{o}'s formula to $|\widehat{x}_s|^2$ leads to
\begin{equation}\label{eq:4}
	\begin{aligned}
		|\widehat{x}_t|^2=&|\widehat{x}_0|^2+2\int_0^t\left\langle\widehat{x}_s,\widehat{b}_s\right\rangle ds+\int _0^t|\widehat{\sigma }_s|^2ds+\sum_{i,j=1}^{\infty}\int_{0}^{t}\left \langle \widehat{g}_s^{(i)},\widehat{g}_s^{(j)}  \right \rangle  d[H_i,H_j]_s\\&+2\int _0^t\left\langle\widehat{x}_s,\widehat{\sigma }_sdW_s\right\rangle +2\sum_{i=1}^{\infty}\int _0^t\left\langle\widehat{x}_{s-},
\widehat{g}^{(i)}_{s-}\right\rangle dH^{(i)}_s.
	\end{aligned}
\end{equation}
Taking expectation on both sides and applying the elementary inequality $2ab\le a^2+b^2$, we can get
\begin{equation}\label{eq:0.5}
	\mathbb{E}\big[|\widehat{x}_t|^2\big]\le\mathbb{E}\Bigg \{|\widehat{x}_0|^2+\int_0^t\big(|\widehat{x}_s|^2+|\widehat{b}_s|^2+|\widehat{\sigma }_s|^2+ \|\widehat{g}_s\|_{l^2(\mathbb{R}^n)}^2\big)ds\Bigg \}.	
\end{equation}
It is easy to verify that
\begin{equation}\label{eq:0.4}
	\int_{0}^{t} |\widehat{x}'_{s}|^2ds=\int_{-\delta }^{t-\delta } |\widehat{x}_{s}|^2ds=\int_{-\delta }^{0} |\widehat{x}_{s}|^2+\int_{0}^{t-\delta } |\widehat{x}_{s}|^2ds\le \delta\displaystyle\sup_{t\in[-\delta,0]}|\widehat{\lambda}_{s}|^2+\int_{0}^{t} |\widehat{x}_{s}|^2ds.
\end{equation}
With the Lipschitz condition in Assumption \ref{ass:2.1}, the inequality \eqref{eq:0.4} and the elementary inequality $(a+b)^2\le 2a^2+2b^2$, we find that
\begin{equation}\label{eq:0.6}
	\begin{aligned}
		\int_0^t|\widehat{b}_s |^2ds=&\int_0^t\big|b(s,x_s,x'_{s })-b(s,\bar{x}_s,\bar{x}'_{s})+b(s,\bar{x}_s,\bar{x}'_{s })-\bar{b}(s,\bar{x}_s,\bar{x}'_{s})\big|^2ds\\
		\le&2\int_0^t\big|b(s,x_s,x'_{s})-b(s,\bar{x}_s,\bar{x}'_{s})\big|^2ds
		+2\int_0^t\big|b(s,\bar{x}_s,\bar{x}'_{s })-\bar{b}(s,\bar{x}_s,\bar{x}'_{s})\big|^2ds\\
		\le&2L^2\int_0^t\big(|\widehat{x}_s|+|\widehat{x}'_{s}|\big)^2ds+2\int_0^t\big|b(s,\bar{x}_s,\bar{x}'_{s})-\bar{b}(s,\bar{x}_s,\bar{x}'_{s })\big|^2ds\\
		\le& 4L^2\int_0^t|\widehat{x}_s|^2ds+4L^2\int_{0 }^{t }|\widehat{x}'_{s}|^2ds
		+2\int_0^t\big|b(s,\bar{x}_s,\bar{x}'_{s })-\bar{b}(s,\bar{x}_s,\bar{x}'_{s})\big|^2ds\\
		\le&8L^2\int_0^t|\widehat{x}_s|^2ds+4\delta L^2\displaystyle\sup_{t\in[-\delta,0]}|\widehat{\lambda}_s|^2
		+2\int_0^t\big|b(s,\bar{x}_s,\bar{x}'_{s })-\bar{b}(s,\bar{x}_s,\bar{x}'_{s })\big|^2ds.
	\end{aligned}
\end{equation}
Likewise,
\begin{equation}\label{eq:0.7}
	\begin{aligned}
\int_0^t|\widehat{\sigma }_s|^2ds\le8L^2\int_0^t|\widehat{x}_s|^2ds+4\delta L^2\displaystyle\sup_{t\in[-\delta,0]}|\widehat{\lambda}_s|^2+2\int_0^t\big|\sigma (s,\bar{x}_s,\bar{x}'_{s})-\bar{\sigma }(s,\bar{x}_s,\bar{x}'_{s })\big|^2ds,
	\end{aligned}
\end{equation}
\begin{equation}\label{eq:0.8}
	\begin{aligned}	
\int_0^t\|\widehat{g}_s\|_{l^2(\mathbb{R}^n)}^2ds\le8L^2\int_0^t|\widehat{x}_s|^2ds+4\delta L^2\displaystyle\sup_{t\in[-\delta,0]}|\widehat{\lambda}_s|^2+2\int_0^t\big\|g(s,\bar{x}_s,\bar{x}'_{s })-\bar{g}(s,\bar{x}_s,\bar{x}'_{s })\big\|_{l^2(\mathbb{R}^n)}^2ds.
	\end{aligned}	
\end{equation}
Putting \eqref{eq:0.6}, \eqref{eq:0.7} and \eqref{eq:0.8} into \eqref{eq:0.5}, we derive
\begin{equation}
	\begin{aligned}
		\mathbb{E}\big[|\widehat{x}_t|^2\big]\le&K\mathbb{E}\Bigg \{\int_0^t|\widehat{x}_s|^2ds+\displaystyle\sup_{t\in[-\delta,0]}|\widehat{\lambda}_s|^2
		+\int_0^t\big|b(s,\bar{x}_s,\bar{x}'_{s })-\bar{b}(s,\bar{x}_s,\bar{x}'_{s })\big|^2ds\\
		&+\int_0^t\big|\sigma (s,\bar{x}_s,\bar{x}'_{s })-\bar{\sigma }(s,\bar{x}_s,\bar{x}'_{s})\big|^2ds
		+\int_0^t\big\|g(s,\bar{x}_s,\bar{x}'_{s })-\bar{g}(s,\bar{x}_s,\bar{x}'_{s })\big\|_{l^2(\mathbb{R}^n)}^2ds\Bigg \}.
	\end{aligned}
\end{equation}
Thus, applying Gronwall’s inequality gives
\begin{equation}\label{eq:0.11}
	\begin{aligned}
		\displaystyle\sup_{t\in[0,T]}\mathbb{E}\big[|\widehat{x}_t|^2\big]\le&K\mathbb{E}\Bigg \{\displaystyle\sup_{t\in[-\delta,0]}|\widehat{\lambda}_s|^2
		+\int_0^T\big|b(s,\bar{x}_s,\bar{x}'_{s })-\bar{b}(s,\bar{x}_s,\bar{x}'_{s })\big|^2ds+\int_0^T\big|\sigma (s,\bar{x}_s,\bar{x}'_{s })-\bar{\sigma }(s,\bar{x}_s,\bar{x}'_{s})\big|^2ds\\
		&+\int_0^T\big\|g(s,\bar{x}_s,\bar{x}'_{s })-\bar{g}(s,\bar{x}_s,\bar{x}'_{s })\big\|_{l^2(\mathbb{R}^n)}^2ds\Bigg \}.
	\end{aligned}
\end{equation}
From \eqref{eq:4}, \eqref{eq:0.6} to \eqref{eq:0.8} and \eqref{eq:0.11}, we apply the Burkholder-Davis-Gundy inequality and derive that
\begin{equation}
	\begin{aligned}
		\mathbb{E}\big[\displaystyle\sup_{t\in[0,T]}|\widehat{x}_t|^2\big]\le\ &K\mathbb{E}\Bigg \{\displaystyle\sup_{t\in[-\delta,0]}|\widehat{\lambda}_s|^2+\int_{0}^{T}\big(|\widehat{x}_s|^2+|\widehat{b}_s|^2+|\widehat{\sigma }_s|^2+ \|\widehat{g}_s\|_{l^2(\mathbb{R}^n)}^2\big)ds\\&\qquad+\displaystyle\sup_{t\in[0,T]}2\int _0^t\left\langle\widehat{x}_s,\widehat{\sigma }_sdW_s\right\rangle +\displaystyle\sup_{t\in[0,T]}2\sum_{i=1}^{\infty}\int _0^t\left\langle\widehat{x}_{s-},
		\widehat{g}^{(i)}_{s-}\right\rangle dH^{(i)}_s   \Bigg \}\\
		\le\ &K\mathbb{E}\Bigg \{\displaystyle\sup_{t\in[-\delta,0]}|\widehat{\lambda}_s|^2+\int_0^T\big|b(s,\bar{x}_s,\bar{x}'_{s })-\bar{b}(s,\bar{x}_s,\bar{x}'_{s })\big|^2ds+\int_0^T\big|\sigma (s,\bar{x}_s,\bar{x}'_{s })-\bar{\sigma }(s,\bar{x}_s,\bar{x}'_{s})\big|^2ds\\
		&\qquad+\int_0^T\big\|g(s,\bar{x}_s,\bar{x}'_{s })-\bar{g}(s,\bar{x}_s,\bar{x}'_{s })\big\|_{l^2(\mathbb{R}^n)}^2ds\Bigg \}+\frac{1}{2}\mathbb{E}\big[\displaystyle\sup_{t\in[0,T]}|\widehat{x}_t|^2\big].
	\end{aligned}
\end{equation}
Then we can easily deduce the desired estimate \eqref{eq:2.3} and we take $(\bar{b}, \bar{\sigma}, \bar{g},\bar{\lambda})=(0,0,0,0)$ such taht  the estimate \eqref{eq:2.2} holds.
For the proof of the existence of uniqueness of the solution, we can also get it directly from the estimate \eqref{eq:2.3} by the method of continuation.
\end{proof}
Secondly, we consider the ABSDEL as follows:
\begin{equation}\label{eq:2.4}
	\left\{\begin{aligned}
		&dy_{t} =f(t, y_{t},z_{t},k_{t},y'_t,z'_t,k'_t) dt+z_{t}d W_{t}+\sum_{i=1}^{\infty }k_{t} ^{(i)}dH_{t}^{(i)}, \quad t \in[0, T],\\
		&y_T =\nu,\\
		&y_t=z_t=k_t=0,\quad t\in (T,T+\delta],
	\end{aligned}\right.
\end{equation}
where $y'_t=\mathbb{E} ^{\mathcal{F}_t }[y_{t+\delta }]$, $z'_t=\mathbb{E} ^{\mathcal{F}_t }[z_{t+\delta }]$, $k'_t=\mathbb{E} ^{\mathcal{F}_t }[k_{t+\delta }]$.

The coefficients $(\nu, f)$ are assumed to satisfy the following assumptions:
\begin{ass}\label{ass:2.2}
	$\nu\in L^2_{\mathcal{F}_T } (\Omega ;\mathbb{R}^n )$ and $f$ is a given random mapping
	\begin{equation}
		\begin{aligned}
			  f:[0,T]\times \Omega \times \mathbb{R}^n \times \mathbb{R}^{n\times d}\times l^2(\mathbb{R}^n) \times \mathbb{R}^n \times \mathbb{R}^{n\times d}\times l^2(\mathbb{R}^n)\to \mathbb{R}^n\nonumber
		\end{aligned}
	\end{equation}
satisfying

	(i)For any $(y,z,k,y',z',k')\in \mathbb{R}^n\times \mathbb{R}^{n\times d}\times l^2(\mathbb{R}^n)\times\mathbb{R}^n\times \mathbb{R}^{n\times d}\times l^2(\mathbb{R}^n)$, $f(\cdot, y, z, k, y', z', k' )$ is $\mathbb{F}$-progressively measurable. Besides, $f(\cdot , 0, 0, 0, 0, 0, 0)\in  L^2_{\mathbb{F} }(0,T;\mathbb{R}^{n} )$.

  (ii)The mapping $f$ is uniformly Lipschitz continuous with respect to $(y,z,k,y',z',k')$, i.e., for any $y,\bar{y},y',\bar{y}'\in \mathbb{R}^n$, $z,\bar{z},z',\bar{z}'\in \mathbb{R}^{n\times d}$, $k,\bar{k},k',\bar{k}'\in l^2(\mathbb{R}^n) $, there exists a constant $L>0$ such that
	\begin{eqnarray}
		\begin{aligned}
		|f&(t , y, z, k, y', z', k')-f(t ,\bar{y} , \bar{z}, \bar{k}, \bar{y}', \bar{z}', \bar{k}')|\\
		&\le L\big(|y-\bar{y} |+|z-\bar{z} |+\|k-\bar{k} \|_{l^2(\mathbb{R}^n)}+|y'-\bar{y}' |+|z'-\bar{z}' |+\|k'-\bar{k}' \|_{l^2(\mathbb{R}^n)}\big).\nonumber
		\end{aligned}
	\end{eqnarray}
\end{ass}
\begin{lem}\label{lem:2.3}
	Under Assumption \ref{ass:2.2}, ABSDEL \eqref{eq:2.4} with coefficients $(\nu, f)$ admits a unique solution $(y(\cdot ),z(\cdot ),k(\cdot ))\in \mathcal{S} ^2_\mathbb{F}(0,T;\mathbb{R}^n )\times  M^2_\mathbb{F}(0,T;\mathbb{R}^{n\times d} )\times L^2_\mathbb{F}(0,T;l^2(\mathbb{R}^n) )$. Moreover, the following estimate holds:
	\begin{eqnarray}\label{eq:2.5}
		\begin{aligned}
			\mathbb E\bigg[\displaystyle\sup_{t\in[0,T]}|y_{t}|^2+\int_{0}^{T}(|z_{t}|^{2}
			+\|k_{t}\|_{l^2(\mathbb{R}^n)}^{2})dt\bigg]
			\le&K\mathbb E\bigg[|\nu|^2+\displaystyle\int_0^T|f(t,0,0,0,0,0,0)|^2dt
			\bigg],
		\end{aligned}
	\end{eqnarray}
where $K$ is a positive constant depending only on $T$ and the Lipschitz constant $L$ of the mapping $f$. Furthermore, suppose that $(\bar{\nu}, \bar{f})$ is another set of coefficients, and assume that $(\bar{y}(\cdot ),\bar{z}(\cdot ),\bar{k}(\cdot ))\in \mathcal{S} ^2_\mathbb{F}(0,T;\mathbb{R}^n )\times  M^2_\mathbb{F}(0,T;\mathbb{R}^{n\times d} )\times L^2_\mathbb{F}(0,T ;l^2(\mathbb{R}^n) )$ is a solution to ABSDEL \eqref{eq:2.4} with coefficients $(\bar{\nu}, \bar{f})$ satisfying Assumption \ref{ass:2.2}.  Then we have the following estimate:
\begin{eqnarray}\label{eq:2.6}
	\begin{aligned}
		&\quad\mathbb{E}\bigg[\displaystyle\sup_{t\in[0,T]}|y_t-\bar{y}_t|^2+\int_{0}^{T}(|z_{t}-\bar{z}_{t}|^{2}
		+\|k_{t}-\bar{k}_{t}\|_{l^2(\mathbb{R}^n)}^{2})dt\bigg]
		\\&\le K\mathbb E\bigg[
		\displaystyle\int_0^T|f(t,\bar{y}_{t},\bar{z}_{t},\bar{k}_{t},\bar{y}'_t,\bar{z}'_t,\bar{k}'_t)-\bar{f}(t,\bar{y}_{t},\bar{z}_{t},\bar{k}_{t},\bar{y}'_t,\bar{z}'_t,\bar{k}'_t)|^2dt+|\nu-\bar{\nu}|^2\bigg],
	\end{aligned}
\end{eqnarray}
where $\bar{y}'_t=\mathbb{E} ^{\mathcal{F}_t }[\bar{y}_{t+\delta }]$, $\bar{z}'_t=\mathbb{E} ^{\mathcal{F}_t }[\bar{z}_{t+\delta }]$, $\bar{k}'_t=\mathbb{E} ^{\mathcal{F}_t }[\bar{k}_{t+\delta }]$ and the positive constant $K$ is similar to the previous one.
\end{lem}
\begin{proof}
	In fact, Theorem 3.2 of Li and Wu \cite{li2014maximum} has also proved the unqiue solvability result of solution to ABSDEL \eqref{eq:2.4} by means of the Fix Point Theorem. Therefore, here we will  only give the proof of the estimate of the solution.
	
Similarly, the constant $K$ can also be changed line to line. Furthermore, to simplify our notation, we denote by
\begin{equation}
	\left\{\begin{aligned}
		&\widehat{y}_s=y_s-\bar{y}_s,\quad \widehat{z}_s=z_s-\bar{z}_s,\quad \widehat{k}_s=k_s-\bar{k}_s, \\
		&\widehat{y}'_s=y'_s-\bar{y}'_s,\quad \widehat{z}'_s=z'_s-\bar{z}'_s,\quad \widehat{k}'_s=k'_s-\bar{k}'_s, \\
		&\widehat{f}_s=f(s,y_s,z_s,k_s,y_s',z_s',k_s')- \bar{f}(s,\bar{y}_s,\bar{z}_s,\bar{k}_s,\bar{y}_s',\bar{z}_s',\bar{k}_s'),\\
		&\widehat{\nu}=\nu-\bar{\nu}.
	\end{aligned}\right.\nonumber
\end{equation}
Firstly, applying the It\^{o}'s formula to $|\widehat{y}_s|^2$ gives
\begin{equation}\label{eq:16}
 	\begin{aligned}
 		&|\widehat{y}_t|^2+\int_{t}^{T}|\widehat{z}_t|^2+\sum_{i,j=1}^{\infty}\int_{t}^{T}\left \langle \widehat{k}_s^{(i)},\widehat{k}_s^{(j)}  \right \rangle  d[H_i,H_j]_s\\
 		=\ &|\widehat{\nu}|^2-2\int_{t}^{T}\left \langle \widehat{y}_s,\widehat{f}_s\right \rangle ds -2\int_{t}^{T}\left \langle \widehat{y}_s,\widehat{z}_sdW_s\right \rangle -2\sum_{i=1}^{\infty}\int_{t}^{T}\left \langle \widehat{y}_s,\widehat{k}^{(i)}_s\right \rangle dH_s^{(i)}.
 	\end{aligned}
 \end{equation}
Taking expectation on both sides, we have
	\begin{align}
		&\quad\mathbb E\bigg[\displaystyle|\widehat{ y}_{t}|^2+\int_{t}^{T}\big(|\widehat{ z}_{s}|^{2}+\|\widehat{k}_{s}\|_{l^2(\mathbb{R}^n)}^{2}\big)ds\bigg ]\nonumber\\
		&\le \mathbb{E}\Bigg \{\big|\widehat{\nu}\big|^2+\int _t^T 2|\widehat{y}_s||\widehat{f}_s|ds \Bigg \}\nonumber\\
		&\le \mathbb{E}\Bigg \{ \big|\widehat{\nu}\big|^2+\frac{1}{\varepsilon } \int _t^T|\widehat{y}_s|^2ds+\varepsilon \int _t^T|\widehat{f}_s|^2ds\Big \}\nonumber\\
		&=\mathbb{E}\Bigg \{ \big|\widehat{\nu}\big|^2+\frac{1}{\varepsilon } \int _t^T|\widehat{y}_s|^2ds+\varepsilon \int _t^T|f(s,y,z,k,y',z',k')-f(s,\bar{y}_s,\bar{z}_s,\bar{k}_s,\bar{y}_s',\bar{z}_s',\bar{k}_s')\\
		&~~~~~~~~ +f(s,\bar{y}_s,\bar{z}_s,\bar{k}_s,\bar{y}_s',\bar{z}_s',\bar{k}_s')-\bar{f}(s,\bar{y}_s,\bar{z}_s,\bar{k}_s,\bar{y}_s',\bar{z}_s',\bar{k}_s')|^2ds\Bigg \}\nonumber\\
		&\le \mathbb{E}\Bigg \{ \big|\widehat{\nu}\big|^2+\frac{1}{\varepsilon } \int _t^T|\widehat{y}_s|^2ds+2\varepsilon \int _t^T|f(s,y,z,k,y',z',k')-f(s,\bar{y}_s,\bar{z}_s,\bar{k}_s,\bar{y}_s',\bar{z}_s',\bar{k}_s')|^2ds\nonumber\\
		&~~~~~~~~+2\varepsilon \int _t^T|f(s,\bar{y}_s,\bar{z}_s,\bar{k}_s,\bar{y}_s',\bar{z}_s',\bar{k}_s')-\bar{f}(s,\bar{y}_s,\bar{z}_s,\bar{k}_s,\bar{y}_s',\bar{z}_s',\bar{k}_s')|^2ds\Bigg \},\nonumber
	\end{align}
where the elementary inequality $2ab\le \frac{1}{\varepsilon }a^2+\varepsilon b^2$ and $(a+b)^2\le 2a^2+2b^2$ for any $a>0$, $b>0$, $\varepsilon>0$ have been used.
Then applying the Lipschitz condition in Assumption \ref{ass:2.2} and the Cauchy-Schwarz inequality yields
\begin{equation}\label{eq:2.14}
	\begin{aligned}
		&\quad\mathbb E\bigg[\displaystyle|\widehat{ y}_{t}|^2+\int_{t}^{T}\big(|\widehat{ z}_{s}|^{2}+\|\widehat{k}_{s}\|_{l^2(\mathbb{R}^n)}^{2}\big)ds\bigg ]\\
		&\le \mathbb{E}\Bigg \{ \big|\widehat{\nu}\big|^2+\frac{1}{\varepsilon } \int _t^T|\widehat{y}_s|^2ds+2\varepsilon \int _t^T|f(s,\bar{y}_s,\bar{z}_s,\bar{k}_s,\bar{y}_s',\bar{z}_s',\bar{k}_s')-\bar{f}(s,\bar{y}_s,\bar{z}_s,\bar{k}_s,\bar{y}_s',\bar{z}_s',\bar{k}_s')|^2ds\\
		&~~~~~~~~+12\varepsilon L^2\int _t^T\Big[|\widehat{y}_s|^2+|\widehat{z}_s|^2+\|\widehat{k}_s\|_{l^2(\mathbb{R}^n)}^2+|\widehat{y}'_s|^2+|\widehat{z}'_s|^2+\|\widehat{k}'_s\|_{l^2(\mathbb{R}^n)}^2\Big]ds\Bigg\}.
	\end{aligned}
\end{equation}
By virtue of the Jensen's inequality and time-shifting transformation, we notice that
\begin{equation}\label{eq:2.15}
	\begin{aligned}
	\mathbb{E}\Big[\int_{t}^{T} |\widehat{y}'_s |^2ds\Big]=\mathbb{E}\Big[\int_{t}^{T} \big|\mathbb{E}\big[\widehat{y}_{s+\delta}|\mathcal{F}_s\big] \big|^2ds\Big]\le\mathbb{E}\Big[\int_{t}^{T }\mathbb{E}\big[|\widehat{y}_{s+\delta}|^2|\mathcal{F}_s\big] ds\Big]=\mathbb{E}\Big[\int_{t+\delta }^{T+\delta}|\widehat{y}_s |^2ds\Big]\le \mathbb{E}\Big[\int_{t}^{T}|\widehat{y}_s |^2ds\Big].
	\end{aligned}
\end{equation}
Similarly, we have
\begin{equation}\label{eq:2.16}
	\mathbb{E}\Big[\int_{t}^{T} |\widehat{z}'_s |^2ds\Big]\le \mathbb{E}\Big[\int_{t}^{T} |\widehat{z}_s |^2ds\Big],
\end{equation}
\begin{equation}\label{eq:2.17}
	\mathbb{E}\Big[\int_{t}^{T} \|\widehat{k}'_s \|_{l^2(\mathbb{R}^n)}^2ds\Big]\le \mathbb{E}\Big[\int_{t}^{T} \|\widehat{k}_s \|_{l^2(\mathbb{R}^n)}^2ds\Big].
\end{equation}
Taking \eqref{eq:2.15}, \eqref{eq:2.16}, \eqref{eq:2.17} into \eqref{eq:2.14}, we can get
\begin{equation}
	\begin{aligned}
		&\quad\mathbb E\bigg[\displaystyle|\widehat{ y}_{t}|^2+\int_{t}^{T}\big(|\widehat{ z}_{s}|^{2}+\|\widehat{k}_{s}\|_{l^2(\mathbb{R}^n)}^{2}\big)ds\bigg ]\\
		&\le \mathbb{E}\Bigg \{ \big|\widehat{\nu}\big|^2+2\varepsilon \int _t^T|f(s,\bar{y}_s,\bar{z}_s,\bar{k}_s,\bar{y}_s',\bar{z}_s',\bar{k}_s')-\bar{f}(s,\bar{y}_s,\bar{z}_s,\bar{k}_s,\bar{y}_s',\bar{z}_s',\bar{k}_s')|^2ds\\
		&\quad+\frac{1}{\varepsilon } \int _t^T|\widehat{y}_s|^2ds+24\varepsilon L^2\int _t^T\Big[|\widehat{y}_s|^2+|\widehat{z}_s|^2+\|\widehat{k}_s\|_{l^2(\mathbb{R}^n)}^2\Big]ds\Bigg \}.
	\end{aligned}
\end{equation}
Selecting $\varepsilon$ small enough such that $24\varepsilon L^2<1$ leads to
\begin{equation}
	\begin{aligned}
		&\mathbb E\bigg[|\widehat{ y}_{t}|^2+\int_{t}^{T}\big(|\widehat{ z}_{s}|^{2}+\|\widehat{k}_{s}\|_{l^2(\mathbb{R}^n)}^{2}\big)ds\bigg ]\\
		\le\ & K\mathbb{E}\Bigg \{ \int_t^T|f(s,\bar{y}_s,\bar{z}_s,\bar{k}_s,\bar{y}_s',\bar{z}_s',\bar{k}_s')-\bar{f}(s,\bar{y}_s,\bar{z}_s,\bar{k}_s,\bar{y}_s',\bar{z}_s',\bar{k}_s')|^2ds
		+\big|\widehat{\nu}\big|^2+ \int _t^T|\widehat{y}_s|^2ds\Bigg\}.
	\end{aligned}
\end{equation}
By using Gronwall’s inequality, we deduce
\begin{equation}\label{eq:24}
	\begin{aligned}
		&\displaystyle\sup_{t\in[0,T]}\mathbb E\big[|\widehat{ y}_{t}|^2\big ]+\mathbb E\bigg[\int_{0}^{T}\big(|\widehat{ z}_{s}|^{2}+\|\widehat{k}_{s}\|_{l^2(\mathbb{R}^n)}^{2}\big)ds\bigg ]\\
	\le\ & K\mathbb{E}\Bigg \{ \int_0^T|f(s,\bar{y}_s,\bar{z}_s,\bar{k}_s,\bar{y}_s',\bar{z}_s',\bar{k}_s')-\bar{f}(s,\bar{y}_s,\bar{z}_s,\bar{k}_s,\bar{y}_s',\bar{z}_s',\bar{k}_s')|^2ds
	+\big|\widehat{\nu}\big|^2\Bigg\}.	
	\end{aligned}
\end{equation}
Based on all of the above analysis, we combine \eqref{eq:16}, \eqref{eq:24} and apply the Burkholder-Davis-Gundy inequality to get
\begin{equation}\label{eq:25}
	\begin{aligned}
		&\mathbb E\bigg[\displaystyle\sup_{t\in[0,T]}|\widehat{ y}_{t}|^2\bigg ]\\
		\le\ & K\mathbb{E}\Bigg\{|\widehat{\nu}|^2+\frac{1}{\varepsilon}\int _0^T|\widehat{y}_s|^2ds+\varepsilon \int _0^T|\widehat{f}_s|^2ds+2\displaystyle\sup_{t\in[0,T]}\bigg| \int_{t}^{T}\left \langle \widehat{y}_s,\widehat{z}_sdW_s\right \rangle\bigg| +2\displaystyle\sup_{t\in[0,T]}\bigg|\sum_{i=1}^{\infty}\int_{t}^{T}\left \langle \widehat{y}_s,\widehat{k}^{(i)}_s\right \rangle dH_s^{(i)}\bigg|\Bigg\}\\
		\le\ &K\mathbb{E}\Bigg\{|\widehat{\nu}|^2+\int_0^T|f(s,\bar{y}_s,\bar{z}_s,\bar{k}_s,\bar{y}_s',\bar{z}_s',\bar{k}_s')-\bar{f}(s,\bar{y}_s,\bar{z}_s,\bar{k}_s,\bar{y}_s',\bar{z}_s',\bar{k}_s')|^2ds\Bigg\}\\&+\frac{1}{2}\mathbb{E}\Big[\displaystyle\sup_{t\in[0,T]}|\widehat{ y}_{t}|^2\Big]+K\mathbb{E}\bigg[\int_{0}^{T}\big(|\widehat{ z}_{s}|^{2}+\|\widehat{k}_{s}\|_{l^2(\mathbb{R}^n)}^{2}\big)ds\bigg].
	\end{aligned}
\end{equation}
Finally, combining \eqref{eq:24} and \eqref{eq:25} leads to the estimate \eqref{eq:2.6}. Then we take $(\bar{\nu},\bar{f})=(0,0)$ to get the estimate \eqref{eq:2.5}. Moreover, for the unique solvability, we can also infer it from the estimate \eqref{eq:2.6} by the method of continuation.
\end{proof}

\section{FBSDELDAs with domination-monotonicity conditions}\label{sec:3}
In this section, we will be committed to studying the FBSDELDA \eqref{eq:1.1}. Similar to the cases of SDEDL \eqref{eq:2.1} and ABSDEL \eqref{eq:2.4}, we still have to make the following assumptions for the coefficients $(\Lambda,\Phi ,\Gamma)$ of FBSDELDA \eqref{eq:1.1}.

\begin{ass}\label{ass:3.1}
(i)For any $x\in \mathbb{R}^n$, $\Phi (x) $ is $\mathcal{F}_{T}$-measurable. Furthermore, for any
$\theta,\theta_{-},\theta_{+} \in \mathbb{R}^{n(2+d)}\times l^2(\mathbb{R}^n)$, $\Gamma (\cdot ,\theta ,\theta_{-},\theta_{+} )$ is $\mathbb{F}$-progressively measurable. Moreover, $(\Lambda(0),\Phi (0),\Gamma (\cdot ,0,0,0) )\in \mathcal{H}[-\delta,T]$; (ii)The mappings $\Phi$, $\Gamma$ are uniformly Lipschitz continuous, i.e., for any $x,\bar{x}\in \mathbb{R}^n$,
$ \theta ,\bar{\theta },\theta_{-},\bar{\theta }_{-},\theta_{+},\bar{\theta}_{+}  \in \mathbb{R}^{n(2+d)}\times l^2(\mathbb{R}^n)$, there exists a constant $L>0$ such that
\begin{equation}
\left\{\begin{aligned}
	&|\Phi (x)-\Phi (\bar{x} )|\le L|x-\bar{x} |,\\
	&|h(t,\theta ,\theta_{-},y_{+},z_{+},k_{+})-h(t,\bar{\theta} ,\bar{ \theta }_{-},\bar{y}_{+},\bar{z}_{+},\bar{k}_{+})|\le L(|\theta -\bar{\theta }|+|\theta_{-}-\bar{\theta }_{-} |+|y_{+}-\bar{y}_{+} |+|z_{+}-\bar{z}_{+} |+\|k_{+}-\bar{k}_{+} \|_{l^2(\mathbb{R}^n)}),\\
	&|f(t,\theta,x_{-},\theta_{+})-f(t,\bar{\theta},\bar{x}_{-},\bar{\theta}_{+})|\le L(|\theta -\bar{\theta }|+|x_{-}-\bar{x}_{-} |+|\theta_{+}-\bar{\theta }_{+} |),
    \end{aligned}\right.\nonumber
\end{equation}
where $h=b,\sigma,g^{(i)}$.
\end{ass}
In addition to the Assumption \ref{ass:3.1} above, we continue to introduce the following domination-monotonicity conditions on the coefficients $(\Lambda,\Phi ,\Gamma)$ for later study.
\begin{ass}\label{ass:3.2}
There exist two constants $\mu\ge 0$, $v\ge0$, a matrix-valued random variable $G\in L_{\mathcal{F}_T }^{\infty }(\Omega ;\mathbb{R}^{\tilde m\times n} )$, and a series of matrix-valued processes $A(\cdot ),\bar{A}(\cdot ), B(\cdot ),\bar{B}(\cdot )\in L_{\mathbb{F} }^{\infty }(0,T;\mathbb{R}^{m\times n} )$, $C(\cdot )
=(C_1(\cdot)),\cdots, C_d(\cdot)),\bar{C}(\cdot )=(\bar C_1(\cdot)),\cdots, \bar C_d(\cdot))\in L_{\mathbb{F} }^{\infty }(0,T;\mathbb{R}^{d\times mn} )$, $
D(\cdot)= (D_j(\cdot)) _{j=1}^{\infty }, \bar D=(\bar{D}_j(\cdot ))_{j=1}^{\infty } \in L_{\mathbb{F} }^{\infty }(0,T;l^2(\mathbb{R}^{m\times n}) )$ (where $\tilde m, m\in\mathbb{N}$ are given and $\bar{A}(t)=\bar{B}(t)=\bar{C}_i(t)=\bar{D}_j(t)=0 (i=1,2,\cdots,d; j=1,2,\cdots,)$  when $t\in[0,\delta]$) such that we have the following conditions:

(i) One of the following two cases holds true. Case $1$: $\mu>0$ and $v=0$. Case $2$: $\mu=0$ and $v>0$.

(ii) (domination condition) For almost all $(t,\omega)\in[0,T]\times\Omega$, and any $x,\bar{x},x_{-},\bar{x}_{-},x_{+},\bar{x}_{+}, y,\bar{y},y_{-},\bar{y}_{-},y_{+},\bar{y}_{+}\in \mathbb{R}^n$, $z,\bar{z},z_{-},\bar{z}_{-},z_{+},\bar{z}_{+}\in \mathbb{R}^{n\times d}$, $k,\bar{k},k_{-},\bar{k}_{-},k_{+},\bar{k}_{+}\in l^2(\mathbb{R}^n)$ (the argument t is suppressed),
\begin{equation}\label{eq:3.1}
	\left\{\begin{aligned}
		&|\Phi (x)-\Phi (\bar{x} )|\le \frac{1}{v}|G\widehat{x} |,\\
		&\int_{0}^{T}\Big(|f(x,y,z,k,x_-,x_+,y_+,z_+,k_+)-f(\bar{x},y,z,k,\bar{x}_{-},\bar{x}_{+},y_{+},z_{+},k_{+})|\Big)dt\\
		&\le \int_{0}^{T}\Big(\frac{1}{v}\big|A\widehat{x} +\mathbb{E}^{\mathcal{F}_t}[\bar{A}_{+}\widehat{x}_{+}]\big|\Big)dt,\\
		&\int_{0}^{T}\Big(|h(x,y,z,k,x_-,y_-,z_-,k_-,y_+,z_+,k_+)-h(x,\bar{y},\bar{z},\bar{k},x_{-},\bar{y}_{-},\bar{z}_{-},\bar{k}_{-},\bar{y}_{+},\bar{z}_{+},\bar{k}_{+})|\Big)dt\\&\le \int_{0}^{T}\bigg(\frac{1}{\mu}\bigg|B\widehat{y}+C\widehat{z}+\sum_{j=1}^{\infty }D_j\widehat{k}^{(j)}+\mathbb{E}^{\mathcal{F}_t}\Big[\bar{B}_{+}\widehat{y}_{+}+\bar{C}_{+}\widehat{z}_{+}+\sum_{j=1}^{\infty }\bar{D}_{j+}\widehat{k}_{+}^{j}\Big]\bigg|\bigg)dt,
	\end{aligned}\right.
\end{equation}
where $h=b, \sigma, g^{(i)}$, and $\widehat{x}=x-\bar{x}$,  $\widehat{y}=y-\bar{y}$, $\widehat{z}=z-\bar{z}$,  $\widehat{k}=k-\bar{k}$ $(i=1,2,3,\cdots)$, etc. $\bar{A}_{+}(\cdot)=\bar{A}(\cdot+\delta)$, $\bar{B}_{+}(\cdot)=\bar{B}(\cdot+\delta)$, etc.

It should be noticed that there is a little abuse of notations in the conditions above, when $\mu=0$ (resp. $v=0$), $1/\mu$ (resp. $1/v$) means $+\infty$. In other words, if $\mu=0$ or $v=0$, the corresponding domination constraints will vanish.

(iii) (monotonicity condition) For almost all $(t,\omega)\in[0,T]\times\Omega$ and any $\theta,\bar{\theta},\theta_{-},\bar{ \theta}_{-},\theta_{+},\bar{ \theta}_{+} \in \mathbb{R}^{n(2+d)}\times l^2(\mathbb{R}^n)$ (the argument t is suppressed),
\begin{equation}\label{eq:3.2}
	\left\{\begin{aligned}
		&\left \langle\Phi (x)-\Phi (\bar{x} ),\widehat{x}\right \rangle\ge v|G\widehat{x}|^2,\\
		&\int_{0}^{T}\Big(\left \langle\Gamma (\theta ,\theta_{-},\theta_{+})-\Gamma (\bar{ \theta} ,\bar{ \theta }_{-},\bar{\theta}_{+}),\widehat{\theta}\right \rangle\Big)dt\\
		&\le \int_{0}^{T}\bigg(-v\big|A\widehat{x} +\mathbb{E}^{\mathcal{F}_t}[\bar{A}_{+}\widehat{x}_{+}]\big|^2-\mu\bigg|B\widehat{y}+C\widehat{z}+\sum_{j=1}^{\infty }D_j\widehat{k}^{(j)}+\mathbb{E}^{\mathcal{F}_t}\Big[\bar{B}_{+}\widehat{y}_{+}+\bar{C}_{+}\widehat{z}_{+}+\sum_{j=1}^{\infty }\bar{D}_{j+}\widehat{k}_{+}^{(j)}\Big]\bigg|^2\bigg)dt,
	\end{aligned}\right.
\end{equation}
where $\widehat{\theta } =\theta -\theta '$ and $\Gamma (t,\theta ,\theta_{-},\theta_{+})$ is given by \eqref{eq:1.2}.            	
\end{ass}
\begin{rmk}\label{rmk:3.1}
(i) In Assumption \ref{ass:3.2}-(ii), the constant $1/\mu$ and $1/v$ can be replaced by $K/\mu$ and $K/v \quad (K>0)$. However, for simplicity, we prefer to omit the constant $K$;\\
(ii) There exists a symmetrical version of Assumption \ref{ass:3.2}-(iii) as follows:

For almost all $(t,\omega)\in[0,T]\times\Omega$ and any $\theta,\bar{\theta},\theta_{-},\bar{ \theta}_{-},\theta_{+},\bar{ \theta}_{+} \in \mathbb{R}^{n(2+d)}\times l^2(\mathbb{R}^n)$ (the argument t is suppressed),
\begin{equation}\label{eq:3.3}
	\left\{\begin{aligned}
		&\left \langle\Phi (x)-\Phi (\bar{x} ),\widehat{x}\right \rangle\le -v|G\widehat{x}|^2,\\
		&\int_{0}^{T}\Big(\left \langle\Gamma (\theta ,\theta_{-},\theta_{+})-\Gamma (\bar{ \theta} ,\bar{ \theta }_{-},\bar{\theta}_{+}),\widehat{\theta}\right \rangle\Big)dt\\
		&\ge \int_{0}^{T}\bigg(v\big|A\widehat{x} +\mathbb{E}^{\mathcal{F}_t}[\bar{A}_{+}\widehat{x}_{+}]\big|^2+\bigg|B\widehat{y}+C\widehat{z}+\sum_{j=1}^{\infty }D_j\widehat{k}^{(j)}+\mathbb{E}^{\mathcal{F}_t}\Big[\bar{B}_{+}\widehat{y}_{+}+\bar{C}_{+}\widehat{z}_{+}+\sum_{j=1}^{\infty }\bar{D}_{j+}\widehat{k}_{+}^{(j)}\Big]\bigg|^2\bigg)dt.
	\end{aligned}\right.
\end{equation}
It is easy to verify the symmetry between the two and we omit the detailed proofs. For similar proofs, we can refer to Yu \cite{yu2022forward}.\\
(iii) In the framework of  Brownian motion only in the diffusion term,  Li, Wang and Wu \cite{li2018indefinite} have studied  a kind of anticipated fully coupled
forward–backward stochastic differential delayed equations, where they introduced  the following monotonicity condition:
\begin{eqnarray}
\begin{aligned}	
		&\int_{0}^{T}\left \langle A (t,x_{t-2\delta}, \lambda_t, \lambda_{t-\delta},
  \lambda_{t+\delta}, x_{t+2\delta},  y_{t+2\delta})-A (t,\bar x_{t-2\delta},  \bar \lambda_t, \bar\lambda_{t-\delta},
 \bar \lambda_{t+\delta}, \bar x_{t+2\delta},\bar y_{t+2\delta}), \lambda_{t}-\bar\lambda_{t}   \right \rangle dt\\
		&\le \int_{0}^{T}\bigg(
-\mu\bigg|B\widehat{y}+D\widehat{z}+\mathbb{E}^{\mathcal{F}_t}\Big[\bar{B}_{+}\widehat{y}_{+}\Big]\bigg|^2\bigg)dt.
  \end{aligned}	
\end{eqnarray}
Compared with this monotonicity condition, our monotonicity condition \eqref{eq:3.2} contains  the terms $ \big|A\widehat{x} +\mathbb{E}^{\mathcal{F}_t}[\bar{A}_{+}\widehat{x}_{+}]\big|^2 $, and $\mathbb{E}^{\mathcal{F}_t}[\bar{D}_{+}\widehat{z}_{+}]$
which will be founded that these conditions are necessary in our application in section 4.

\end{rmk}
For the convenience of later use, we would like to give some notations as follows (the argument $t$ is suppressed):
{\small\begin{equation}\label{eq:6.4}
	\left\{\begin{aligned}
		&P(x)=Ax +\mathbb{E}^{\mathcal{F}_t}[\bar{A}_{+}x_{+}],\\
		&P_{-}(x)=A_{-}x_{-} +\mathbb{E}^{\mathcal{G}_{t}}[\bar{A}x],\\
		&P(\widehat{x})=A\widehat{x} +\mathbb{E}^{\mathcal{F}_t}[\bar{A}_{+}\widehat{x}_{+}],\\
		&P_{-}(\widehat{x})=A_{-}\widehat{x}_{-} +\mathbb{E}^{\mathcal{G}_{t}}[\bar{A}\widehat{x}],\\
		&Q(y,z,k)=By+Cz+\sum_{j=1}^{\infty }D_jk^{(j)}+\mathbb{E}^{\mathcal{F}_t}\Big[\bar{B}_{+}y_{+}+\bar{C}_{+}z_{+}+\sum_{j=1}^{\infty }\bar{D}_{j+}k_{+}^{j}\Big],\\
		&Q_{-}(y,z,k)=B_{-}y_{-}+C_{-}z_{-}+\sum_{j=1}^{\infty }D_{j-}k_{-}^{(j)}+\mathbb{E}^{\mathcal{G}_{t}}\Big[\bar{B}y+\bar{C}z+\sum_{j=1}^{\infty }\bar{D}_{j}k^{j}\Big],\\
		&Q(\widehat{y},\widehat{z},\widehat{k})=B\widehat{y}+C\widehat{z}+\sum_{j=1}^{\infty}D_j\widehat{k}^{(j)}+\mathbb{E}^{\mathcal{F}_t}\Big[\bar{B}_{+}\widehat{y}_{+}+\bar{C}_{+}\widehat{z}_{+}+\sum_{j=1}^{\infty }\bar{D}_{j+}\widehat{k}_{+}^{j}\Big],\\
		&Q_{-}(\widehat{y},\widehat{z},\widehat{k})=B_{-}\widehat{y}_{-}+C_{-}\widehat{z}_{-}+\sum_{j=1}^{\infty}D_{j-}\widehat{k}_{-}^{(j)}+\mathbb{E}^{\mathcal{G}_{t}}\Big[\bar{B}\widehat{y}+\bar{C}\widehat{z}+\sum_{j=1}^{\infty }\bar{D}_{j}\widehat{k}^{j}\Big],
	\end{aligned}\right.
\end{equation}}
where $x_{-}(t)=x(t-\delta)$, $x_{+}(t)=x(t+\delta)$, $A_{-}(t)=A(t-\delta)$, $A_{+}(t)=A(t+\delta)$, etc.

Now, we give the main results of this section.
\begin{thm}\label{thm:3.1}
Let $(\Lambda,\Phi ,\Gamma)$ be a set of coefficients satisfying Assumption \ref{ass:3.1} and Assumption \ref{ass:3.2}. Then FBSDELDA \eqref{eq:1.1} admits a unique solution $\theta (\cdot )\in N_{\mathbb{F} }^2(0 ,T;\mathbb{R}^{n(2+d)}\times l^2(\mathbb{R}^n) )$. Moreover, we have the following estimate:
\begin{equation}\label{eq:3.4}
\mathbb E\bigg[\displaystyle \sup_{t\in [0,T]} |x_t|^2+\displaystyle \sup_{t\in [0,T]} |y_t|^2+\int_{0}^{T}|z_t|^2dt+\int_{0}^{T}\|k_t\|_{l^2(\mathbb{R}^n)}^2dt\bigg ]\le K\mathbb{E}[\mathrm{I}],
\end{equation}
where
\begin{eqnarray}
	\begin{aligned}	
\mathrm{I}=&|\Phi (0)|^2+\int_{0}^{T}|b(t,0,0,0,0,0)|^2 dt+\int_{0}^{T}|\sigma (t,0,0,0,0,0)|^2 dt+\int_{0}^{T}\|g(t,0,0,0,0,0)\|_{l^2(\mathbb{R}^n)}^2 dt\\
&+\int_{0}^{T}|f(t,0,0,0)|^2 dt+\displaystyle \sup_{t\in [-\delta,0]} |\lambda _t|^2 +\displaystyle \sup_{t\in [-\delta,0]}|\mu_t|^2+\int_{-\delta}^{0}|\rho_t|^2dt+\int_{-\delta}^{0}\|\varsigma_t\|_{l^2(\mathbb{R}^n)}^2dt,
    \end{aligned}
\end{eqnarray}
and $K$ is a positive constant depending only on $T$, the Lipschitz constants, $\mu$, $v$ and the bound of all $G$, $A(\cdot)$, $\bar{A}(\cdot)$, $B(\cdot)$, $\bar{B}(\cdot)$, $C(\cdot)$, $\bar{C}(\cdot)$, $D_j(\cdot)$, $\bar{D}_j(\cdot)$ $(j=1,2,\cdots)$, $S(\cdot)$, $\bar{S}(\cdot)$. Furthermore, let $(\bar{\Lambda},\bar{ \Phi },\bar{ \Gamma})$ be another set of coefficients, and $ \bar{ \theta }(\cdot )\in N_{\mathbb{F} }^2(0 ,T  ;\mathbb{R}^{n(2+d)}\times l^2(\mathbb{R}^n) )$ be a solution to FBSDELDA \eqref{eq:1.1} with the coefficients $(\bar{\Lambda},\bar{ \Phi },\bar{ \Gamma})$. Then we continue to assume that $\big(\bar{\Lambda}(\cdot),\bar{ \Phi }(\bar{x}_T ),\bar{ \Gamma}(\cdot ,\bar{\theta }(\cdot ),\bar{\theta }_{-}(\cdot ),\bar{\theta }_{+}(\cdot )\big)\in \mathcal{H}[-\delta,T]$. Then the following estimate holds:
\begin{equation}\label{eq:3.6}
	\mathbb E\bigg[\displaystyle \sup_{t\in [0,T]} |\widehat{x}_t|^2+\displaystyle \sup_{t\in [0,T]} |\widehat{y}_t|^2+\int_{0}^{T}|\widehat{z}_t|^2dt+\int_{0}^{T}\|\widehat{k}_t\|_{l^2(\mathbb{R}^n)}^2dt\bigg ]\le K\mathbb{E}[\widehat{\mathrm{I}}],
\end{equation}
where we denote $\widehat{x}=x-\bar{x}$,  $\widehat{y}=y-\bar{y}$, $\widehat{z}=z-\bar{z}$,  $\widehat{k}=k-\bar{k}$ $(i=1,2,3,\cdots)$, etc, and
\begin{eqnarray}
	\begin{aligned}	
		\widehat{\mathrm{I}}=&|\Phi (\bar{x}_T)-\bar{\Phi} (\bar{x}_T)|^2+\int_{0}^{T}|b(t,\bar{\theta}(t),\bar{\theta}_{-}(t),\bar{y}_{+}(t),\bar{z}_{+}(t),\bar{k}_{+}(t))-\bar{b}(t,\bar{\theta}(t),\bar{\theta}_{-}(t),\bar{y}_{+}(t),\bar{z}_{+}(t),\bar{k}_{+}(t))|^2 dt\\&+\int_{0}^{T}|\sigma(t,\bar{\theta}(t),\bar{\theta}_{-}(t),\bar{y}_{+}(t),\bar{z}_{+}(t),\bar{k}_{+}(t))-\bar{\sigma}(t,\bar{\theta}(t),\bar{\theta}_{-}(t),\bar{y}_{+}(t),\bar{z}_{+}(t),\bar{k}_{+}(t))|^2 dt\\&+\int_{0}^{T}\|g(t,\bar{\theta}(t),\bar{\theta}_{-}(t),\bar{y}_{+}(t),\bar{z}_{+}(t),\bar{k}_{+}(t))-\bar{g}(t,\bar{\theta}(t),\bar{\theta}_{-}(t),\bar{y}_{+}(t),\bar{z}_{+}(t),\bar{k}_{+}(t))\|_{l^2(\mathbb{R}^n)}^2
		\\&+\int_{0}^{T}|f(t,\bar{\theta}(t),\bar{x}_{-}(t),\bar{\theta}_{+}(t))-\bar{f}(t,\bar{\theta}(t),\bar{x}_{-}(t),\bar{\theta}_{+}(t))|^2 dt\\&+\displaystyle \sup_{t\in [-\delta,0]} |\widehat{\lambda}_t|^2 +\displaystyle \sup_{t\in [-\delta,0]}|\widehat{\mu}_t|^2+\int_{-\delta}^{0}|\widehat{\rho}_t|^2dt+\int_{-\delta}^{0}\|\widehat{\varsigma}_t\|_{l^2(\mathbb{R}^n)}^2dt,
	\end{aligned}
\end{eqnarray}
and $K$ is the same constant as in \eqref{eq:3.4}.
\end{thm}
Next, we are devoted to proving Theorem \ref{thm:3.1}. Due to the symmetry of monotonicity condition \eqref{eq:3.2} and \eqref{eq:3.3}, we only give the detailed proofs under monotonicity condition \eqref{eq:3.2}.

For any $\big(\pi(\cdot) ,\eta ,\rho (\cdot )\big)\in \mathcal{H} [-\delta,T]$ with $\pi(\cdot)=\big(\xi(\cdot)^\top,\vartheta(\cdot)^\top,\tau(\cdot)^\top,\chi(\cdot)^\top\big)^\top$ and $\rho (\cdot )=(\varphi (\cdot )^\top ,\psi (\cdot )^\top,\gamma  (\cdot )^\top,\beta  (\cdot )^\top)^\top$ where $\beta (\cdot )=(\beta ^{(1)}(\cdot)^\top,\beta ^{(2)}(\cdot)^\top,\cdots )^\top$, we continue to introduce a family of FBSDELDAs parameterized by $\alpha\in[0,1]$ as follows:
  \begin{equation}\label{eq:3.8}
	\left\{\begin{aligned}
		dx^{\alpha}(t) =&\big[b^{\alpha}(t,\theta^{\alpha}(t), \theta_{-}^{\alpha}(t),y_{+}^{\alpha}(t),z_{+}^{\alpha}(t),k_{+}^{\alpha}(t))+\psi(t)\big]dt\\&+\big[\sigma^{\alpha}(t,\theta^{\alpha}(t), \theta_{-}^{\alpha}(t),y_{+}^{\alpha}(t),z_{+}^{\alpha}(t),k_{+}^{\alpha}(t))+\gamma(t)\big]d W(t)\\&+\sum_{i=1}^{\infty }\big[g^{(i)\alpha}  (t,\theta^{\alpha}(t-), \theta_{-}^{\alpha}(t-),y_{+}^{\alpha}(t-),z_{+}^{\alpha}(t),k_{+}^{\alpha}(t))+\beta^{(i)}(t)\big]dH^{(i)}(t),\quad t \in[0, T],    \\
		dy^{\alpha}(t)=&\big[f^{\alpha}(t,\theta^{\alpha}(t), x_{-}^{\alpha}(t),\theta_{+}^{\alpha}(t))+\varphi(t)\big]dt+z^{\alpha}(t)d W(t)+\sum_{i=1}^{\infty }k^{(i)\alpha}(t)dH^{(i)}(t), \quad t \in[0, T],\\
		x^{\alpha}(t) =&\lambda^\alpha(t)+\xi(t),\quad y^{\alpha}(t) =\mu^\alpha(t)+\vartheta(t),\\ z^{\alpha}(t) =&\rho^\alpha(t)+\tau(t),\quad k^{\alpha}(t) =\varsigma^\alpha(t)+\chi(t),\quad t \in[-\delta , 0],\\
		y^{\alpha}(T) =&\Phi^{\alpha} (x^{\alpha}(T))+\eta,\\
		x^{\alpha}(t) =&y^{\alpha}(t)=z^{\alpha}(t)=k^{\alpha}(t)=0,\quad t\in(T,T+\delta],
	\end{aligned}\right.
\end{equation}
where $\theta ^\alpha (t)=(x^{\alpha }(t)^{\top},y^{\alpha}(t)^{\top},z^{\alpha }(t)^{\top},k^{\alpha }(t)^{\top})^{\top}$ with  $k^\alpha(t) :=(k^{(1)\alpha }(t)^{\top},k^{(2)\alpha}(t)^{\top},\cdots)^{\top}$, $\theta_{-}^{\alpha}(t)=(x_{-}^{\alpha}(t)^{\top},y_{-}^{\alpha}(t)^{\top},z_{-}^{\alpha} (t)^{\top},\\k_{-}^{\alpha}(t)^{\top})^{\top }$, $\theta_{+}^{\alpha}(t)=(x_{+}^{\alpha}(t)^{\top},y_{+}^{\alpha}(t)^{\top},z_{+}^{\alpha} (t)^{\top},k_{+}^{\alpha}(t)^{\top})^{\top }$, and $\theta^{\alpha}(t-)$, $\theta_{-}^{\alpha}(t-)$ have the similar definition as that in LBSDELDA \eqref{eq:1.1} which is to say that $t-$ only works on $x$, $x_{-}$ and $y$, $y_{-}$. For any $(t,\omega,\theta,\theta_{-},\theta_{+})\in[-\delta,T]\times\Omega\times\big(\mathbb{R}^{n(2+d)}\times l^2(\mathbb{R}^n)\big)\times\big(\mathbb{R}^{n(2+d)}\times l^2(\mathbb{R}^n)\big)\times \big(\mathbb{R}^{n(2+d)}\times l^2(\mathbb{R}^n)\big)$,
\begin{equation}\label{eq:3.9}
	\left\{\begin{aligned}
		&\Phi ^\alpha (x)=\alpha \Phi (x)+(1-\alpha)vG^{\top}Gx,\\
		&\lambda^{\alpha}(t)=\lambda(t),\quad \mu^{\alpha}(t)=\mu(t),\quad \rho^{\alpha}(t)=\rho(t),\quad \varsigma^{\alpha}(t)=\varsigma(t),\\
		&f^\alpha (t,\theta ,x_{-},\theta_{+})=\alpha f(t,\theta ,x_{-},\theta_{+})-(1-\alpha)v\bigg[A(t)^{\top}P(t,x)+\bar{A}(t)^{\top}P_{-}(t,x)\bigg],\\
		&b^\alpha (t,\theta ,\theta_{-},y_{+},z_{+},k_{+})=\alpha b(t,\theta ,\theta_{-},y_{+},z_{+},k_{+})-(1-\alpha)\mu\bigg[ B(t)^\top Q(t,y,z,k) +\bar{B}(t)^{\top}Q_{-}(t,y,z,k)\bigg],\\
		&\sigma ^\alpha (t,\theta ,\theta_{-},y_{+},z_{+},k_{+})=\alpha \sigma (t,\theta ,\theta_{-},y_{+},z_{+},k_{+})-(1-\alpha)\mu \bigg[C(t)^\top Q(t,y,z,k) +\bar{C}(t)^{\top}Q_{-}(t,y,z,k)\bigg],\\
		&g^{(i)\alpha} (t,\theta ,\theta_{-},y_{+},z_{+},k_{+})=\alpha g^{(i)}(t,\theta ,\theta_{-},y_{+},z_{+},k_{+})-(1-\alpha)\mu\bigg[D_{i}(t)^\top Q(t,y,z,k) +\bar{D}_{i}(t)^{\top}Q_{-}(t,y,z,k)\bigg],
	\end{aligned}\right.
\end{equation}
where $P(t,x)$, $P_{-}(t,x)$, $Q(t,y,z,k)$, $Q_{-}(t,y,z,k)$ are defined by \eqref{eq:6.4}.
Similarly, we continue to denote $\Gamma ^\alpha (\cdot ):=(f^\alpha(\cdot )^\top ,b^\alpha(\cdot )^\top ,\sigma ^\alpha(\cdot )^\top ,g^\alpha(\cdot )^\top )^{\top}$ with $g^\alpha(\cdot )^\top :=(g^{(1)\alpha }(\cdot )^\top,g^{(2)\alpha }(\cdot )^\top,\cdots )^\top$.

Without loss of generality, we assume that the Lipschitz constants of the coefficients $(\Phi,\Gamma)$ are larger than
\begin{equation}
\begin{aligned}
\text{max}\left\{\mu,v\right\} \bigg (&|G|+\|A(\cdot)\|_{L_{\mathbb{F} }^{\infty }(0,T;\mathbb{R}^{m\times n} )}+\|\bar{A}(\cdot)\|
_{L_{\mathbb{F} }^{\infty }(0,T;\mathbb{R}^{m\times n} )}+\|B(\cdot)\|_{L_{\mathbb{F} }^{\infty }(0,T;\mathbb{R}^{m\times n} )}
+\|\bar{B}(\cdot)\|_{L_{\mathbb{F} }^{\infty }(0,T;\mathbb{R}^{m\times n} )}\\&+\|C(\cdot)\|_{L_{\mathbb{F} }^{\infty }(0,T;\mathbb{R}^{d\times mn} )}+\|\bar{C}(\cdot)\|_{L_{\mathbb{F} }^{\infty }(0,T;\mathbb{R}^{d\times mn} )}+\sum_{j=1}^{\infty}\|D_j(\cdot)\|_{L_{\mathbb{F} }^{\infty }(0,T;\mathbb{R}^{m\times n} )}+\sum_{j=1}^{\infty}\|\bar{D}_j(\cdot)\|_{L_{\mathbb{F} }^{\infty }(0,T;\mathbb{R}^{m\times n} )}\bigg)^2,\nonumber
\end{aligned}
\end{equation}
and the constant $\mu$ and $v$ in Assumption \ref{ass:3.2}-(i) satisfy the following condition:
\begin{equation}
	\begin{aligned}
	(\frac{1}{\mu })^2,(\frac{1}{v})^2\ge \text{max}\Big\{&|G|,\|A(\cdot)\|_{L_{\mathbb{F} }^{\infty }(0,T;\mathbb{R}^{m\times n} )}, \|\bar{A}(\cdot)\|_{L_{\mathbb{F} }^{\infty }(0,T;\mathbb{R}^{m\times n} )}, \|B(\cdot)\|_{L_{\mathbb{F} }^{\infty }(0,T;\mathbb{R}^{m\times n} )}, \|\bar{B}(\cdot)\|_{L_{\mathbb{F} }^{\infty }(0,T;\mathbb{R}^{m\times n} )},\\ & \|C(\cdot)\|_{L_{\mathbb{F} }^{\infty }(0,T;\mathbb{R}^{d\times mn} )},\|\bar{C}(\cdot)\|_{L_{\mathbb{F} }^{\infty }(0,T;\mathbb{R}^{d\times mn} )},\|D_j(\cdot)\|_{L_{\mathbb{F} }^{\infty }(0,T;\mathbb{R}^{m\times n} )},\|\bar{D}_j(\cdot)\|_{L_{\mathbb{F} }^{\infty }(0,T;\mathbb{R}^{
m\times n} )} \Big\}.\nonumber
    \end{aligned}
\end{equation}
Then we can easily verify that for any $\alpha\in[0,1]$, the new coefficients $(\Lambda^\alpha,\Phi^\alpha,\Gamma^\alpha)$ also satisfy Assumption \ref{ass:3.1} and Assumption \ref{ass:3.2}  with the same Lipschitz constants, $\mu$, $v$, $G$, $A(\cdot)$, $\bar{A}(\cdot)$, $B(\cdot)$, $\bar{B}(\cdot)$, $C(\cdot)$, $\bar{C}(\cdot)$, $D_j(\cdot)$ $\bar{D}_j(\cdot)$ $(j=1,2,\cdots)$ as the original coefficients $(\Lambda,\Phi ,\Gamma)$.

Obviously, when $\alpha=0$, FBSDELDA \eqref{eq:3.8} can be rewritten in the following form:
\begin{equation}\label{eq:3.10}
	\left\{\begin{aligned}
		dx^0(t)=&\bigg \{-\mu \Big[B(t)^\top  Q(t,y_t^0,z_t^0,k_t^0)+\bar{B}(t)^{\top} Q_{-}(t,y_t^0,z_t^0,k_t^0)\Big]+\psi(t)\bigg \}dt
		\\&+\bigg \{-\mu\Big[ C(t)^\top Q(t,y_t^0,z_t^0,k_t^0)+\bar{C}(t)^{\top} Q_{-}(t,y_t^0,z_t^0,k_t^0)\Big]+\gamma(t)\bigg \}dW(t)
		\\&+\sum_{i=1}^{\infty} \bigg \{-\mu \Big[D_i(t)^\top Q(t,y_{t-}^0,z_t^0,k_t^0)+ \bar{D}_i(t)^{\top}Q_{-}(t,y_{t-}^0,z_t^0,k_t^0)\Big]+\beta^{(i)}(t)\bigg \}dH^{(i)}(t),\\
		dy^0(t)=&\bigg(-v\Big[A(t)^{\top}P(t,x_t^0)+\bar{A}(t)^{\top}P_{-}(t,x_t^0)\Big]+\varphi(t)\bigg)dt\\&+z^0(t)d W(t)+\sum_{i=1}^{\infty }k ^{(i)0}(t)dH^{(i)}(t),\quad t\in [0,T],\\
		x^{0}(t) =&\lambda^0(t)+\xi(t),\quad y^{0}(t) =\mu^0(t)+\vartheta(t),\quad
		z^{0}(t) =\rho^0(t)+\tau(t),\quad k^{0}(t) =\varsigma^0(t)+\chi(t),\quad t \in[-\delta , 0],\\
		y^0(T)=&vG^{\top}Gx^0(T)+\eta,\quad
		x^{0}(t)=y^{0}(t)=z^{0}(t)=k^{0}(t)=0,\quad t\in(T,T+\delta].
	\end{aligned}\right.
\end{equation}
In fact, we can easily find that FBSDELDA \eqref{eq:3.10} is in a decoupled form. In detail, when Assumption \ref{ass:3.2}-(i)-Case 1 holds (i.e., $\mu>0$ and $v=0$), we can firstly solve $(y^0(\cdot),z^0(\cdot),k^0(\cdot))$ from the backward equation, then substitute $(y^0(\cdot),z^0(\cdot),k^0(\cdot))$ into the forward equation and solve $x^0(\cdot)$. Similarly, when Assumption \ref{ass:3.2}-(i)-Case 2 holds (i.e., $\mu=0$ and $v>0$), we can firstly solve the forward equation and then the backward equation. In short, under Assumption \ref{ass:3.1} and Assumption \ref{ass:3.2}, FBSDELDA \eqref{eq:3.10} admits a unique solution $\theta^0 (\cdot )\in N_{\mathbb{F} }^2(0 ,T;\mathbb{R}^{n(2+d)}\times l^2(\mathbb{R}^n) )$.

It is clear that when $\alpha=1$ and $(\pi(\cdot) ,\eta ,\rho (\cdot ))$ vanish, FBSDELDA \eqref{eq:3.8} and FBSDELDA \eqref{eq:1.1} are identical.  Next, we will illustrate that if for some $\alpha_{0}\in[0,1)$, FBSDELDA \eqref{eq:3.8} is uniquely solvable for any $\big(\pi(\cdot) ,\eta ,\rho (\cdot )\big)\in \mathcal{H} [-\delta,T]$, then there exists a fixed step length $\delta_0>0$ such that the same conclusion still holds for any $\alpha\in[\alpha_{0},\alpha_{0}+\delta_0]$. As long as this has been proved to be true, we can gradually increase the parameter $\alpha$ until $\alpha=1$. This method is called the method of continuation which is introduced initially by Hu and Peng \cite{hu1995solution}.

For this goal, we shall first establish a priori estimate for the solution of FBSDELDA \eqref{eq:3.8} which plays an improtant role in the subsequent proofs.
\begin{lem}\label{lem:3.2}
Let the given coefficients $(\Lambda,\Phi ,\Gamma)$ satisfy Assumption \ref{ass:3.1} and Assumption \ref{ass:3.2}. Let $\alpha\in[0,1]$, $\big(\pi(\cdot) ,\eta ,\rho (\cdot )\big)$, $\big(\bar{\pi}(\cdot) ,\bar{\eta} ,\bar{\rho} (\cdot )\big)\in \mathcal{H} [-\delta,T]$. Assume that $\theta(\cdot )\in N_{\mathbb{F} }^2(0  ,T;\mathbb{R}^{n(2+d)} \times l^2(\mathbb{R}^n))$ is the solution of FBSDELDA \eqref{eq:3.8} with the coefficients $(\Lambda^\alpha+\pi,\Phi^\alpha+\eta ,\Gamma^\alpha+\rho)$ and $\bar{\theta}(\cdot )\in N_{\mathbb{F} }^2(0 ,T;\mathbb{R}^{n(2+d)} \times l^2(\mathbb{R}^n))$ is also the solution to FBSDELDA \eqref{eq:3.8} with the coefficients $(\Lambda^\alpha+\bar{\pi},\Phi^\alpha+\bar{\eta} ,\Gamma^\alpha+\bar{\rho})$. Then the following estimate holds:
\begin{equation}\label{eq:3.11}
	\mathbb E\bigg[\displaystyle \sup_{t\in [0,T]} |\widehat{x}_t|^2+\displaystyle \sup_{t\in [0,T]} |\widehat{y}_t|^2+\int_{0}^{T}|\widehat{z}_t|^2dt+\int_{0}^{T}\|\widehat{k}_t\|_{l^2(\mathbb{R}^n)}^2dt\bigg ]\le K\mathbb{E}[\widehat{\mathrm{J}}],
\end{equation}
where we denote
\begin{equation}\label{eq:3.12}
	\begin{aligned}
        \widehat{\mathrm{J} }=&|\widehat{\eta } |^2+\int_{0}^{T}|\widehat{\varphi }_t |^2dt+\int_{0}^{T}|\widehat{\psi }_t |^2dt+\int_{0}^{T}|\widehat{\gamma}_t |^2dt+\int_{0}^{T}\|\widehat{\beta  }_t \|_{l^2(\mathbb{R}^n)}^2dt\\
        &+\displaystyle \sup_{t\in [-\delta,0]} |\widehat{\xi}_t|^2+\displaystyle \sup_{t\in [-\delta,0]}|\widehat{\vartheta} _t|^2+\int_{-\delta}^{0}|\widehat{\tau} _t|^2dt+\int_{-\delta}^{0}\|\widehat{\chi} _t\|_{l^2(\mathbb{R}^n)}^2dt ,
 	\end{aligned}
\end{equation}
and $\widehat{\xi }=\xi -\bar{\xi }$, $\widehat{\varphi }=\varphi -\bar{\varphi }$, etc. Here $K$ is a positive constant that satisfies the same conditions as the constant $K$ in Theorem \ref{thm:3.1}.
\end{lem}
\begin{proof}
In the following proofs, the argument $t$ is suppressed for simplicity. Besides, it should be noted that the positive constant $K$ could be changed line to line.

By the estimate \eqref{eq:2.3} in Lemma \ref{lem:2.2}, we have
\begin{equation}\label{eq:3.13}
	\begin{aligned}
		\mathbb{E}\bigg[\displaystyle \sup_{t\in [ 0,T]} |\widehat{x}|^2\bigg ]&\le K\mathbb {E}\Bigg \{
		\int_{0}^{T}\bigg |\alpha \big (b(\bar{x},y,z,k,\bar{x}_{-},y_{-},z_{-},k_{-},y_{+},z_{+},k_{+} )-b(\bar{\theta},\bar{\theta}_{-},\bar{y}_{+},\bar{z}_{+},\bar{k}_{+})\big )\\
		&\qquad-(1-\alpha)\mu\Big[ B^\top Q(\widehat{y},\widehat{z},\widehat{k})+\bar{B}^\top Q_{-}(\widehat{y},\widehat{z},\widehat{k})\Big]+\widehat{\psi} \bigg |^2dt\\
		&\quad+\int_{0}^{T}\Bigg |\alpha \big (\sigma(\bar{x},y,z,k,\bar{x}_{-},y_{-},z_{-},k_{-},y_{+},z_{+},k_{+} )-\sigma(\bar{\theta},\bar{\theta}_{-},\bar{y}_{+},\bar{z}_{+},\bar{k}_{+})\big )\\
		&\qquad-(1-\alpha)\mu\Big[C^\top Q(\widehat{y},\widehat{z},\widehat{k})+\bar{C}^\top Q_{-}(\widehat{y},\widehat{z},\widehat{k})\Big]+\widehat{\gamma} \Bigg |^2dt\\
		&\quad+ \int_{0}^{T} \Bigg \|\alpha \big (g(\bar{x},y,z,k,\bar{x}_{-},y_{-},z_{-},k_{-},y_{+},z_{+},k_{+} )-g(\bar{\theta},\bar{\theta}_{-},\bar{y}_{+},\bar{z}_{+},\bar{k}_{+})\big )\\
		&\qquad-(1-\alpha)\mu \Big[ D_{i}^\top Q(\widehat{y},\widehat{z},\widehat{k})+\bar{D}_{i}^\top Q_{-}(\widehat{y},\widehat{z},\widehat{k})\Big]+\widehat{\beta} \Bigg \|_{l^2(\mathbb{R}^n)}^2dt +\displaystyle \sup_{t\in [-\delta,0]} |\widehat{\xi}|^2\Bigg\}.
	\end{aligned}
\end{equation}
Similarly, by applying the estimate \eqref{eq:2.6} in Lemma \ref{lem:2.3}, we can derive
\begin{equation}\label{eq:3.14}
	\begin{aligned}
		&\quad\mathbb E\bigg[\displaystyle \sup_{t\in [0,T]} |\widehat{y}|^2+\int_{0}^{T}|\widehat{z}|^2dt+\int_{0}^{T}\|\widehat{k}\|_{l^2(\mathbb{R}^n)}^2dt\bigg ]\\
		&\le K\mathbb E\Bigg\{\Big |\alpha \big (\Phi (x_T)-\Phi (\bar{x }_T )\big )+(1-\alpha )vG^\top G\hat{x}_T+\widehat{\eta }  \Big |^2\\
		&~~~+\int_{0}^{T}\bigg |\alpha \big (f(x,\bar{y},\bar{z},\bar{k},x_{-},x_{+},\bar{y}_{+},\bar{z}_{+},\bar{k}_{+})-f(\bar{\theta},\bar{x}_{-},\bar{\theta}_{+})\big )\\
		&~~~~~~-(1-\alpha )v\Big[A^\top P(\widehat{x})+\bar{A}^\top P_{-}(\widehat{x})\Big]+\widehat{\varphi }\bigg |^2dt\Bigg\}.
	\end{aligned}
\end{equation}
Furthermore, applying It\^{o }'s formula to $\left \langle \widehat{x}(\cdot ),\widehat{y}(\cdot )  \right \rangle$ yields
\begin{equation}\label{eq:3.15}
	\begin{aligned}
		&\quad\mathbb{E}\Bigg \{(1-\alpha)v|G\widehat{x}_T |^2
		+\alpha \left \langle \Phi (x_T)-\Phi (\bar{x}_T),\widehat{x}_T  \right \rangle+(1-\alpha )\mu \int_{0}^{T}\big |Q(\widehat{y},\widehat{z},\widehat{k})\big |^2dt\\
		&\qquad+(1-\alpha )v\int_{0}^{T}\big|P(\widehat{x})\big|^2dt-\alpha \int_{0}^{T}\left \langle \Gamma (\theta ,\theta_{-},\theta_{+})-\Gamma(t,\bar{\theta} ,\bar{\theta}_{-},\bar{\theta}_{+}) ,\widehat{\theta} \right \rangle dt\Bigg \}\\
		&=\mathbb{E}\Bigg \{-\left \langle \widehat{\eta},\widehat{x}_T  \right \rangle +\left \langle \widehat{\xi }_0,\widehat{\vartheta}_0  \right \rangle+\int_{0}^{T}\Big [\left \langle \widehat{\varphi},\widehat{x}\right \rangle +\left \langle \widehat{\psi},\widehat{y}\right \rangle + \left \langle \widehat{\gamma },\widehat{z}  \right \rangle +\left \langle \widehat{\beta },\widehat{k}  \right \rangle \Big ]dt\Bigg \},
	\end{aligned}
\end{equation}
where the assumption that $\bar{A}=\bar{B}=\bar{C}=\bar{D}_j=0$, when $t\in[0,\delta]$ has been used.
Therefore, combining with the monotonicity conditions in Assumption \ref{ass:3.2}-(iii), \eqref{eq:3.15} is reduced to
\begin{equation}\label{eq:3.16}
	\begin{aligned}
		&\quad\mathbb{E}\Bigg \{ v|G\widehat{x}_T |^2+v\int_{0}^{T}\big|P(\widehat{x})\big|^2dt+\mu \int_{0}^{T}\big |Q(\widehat{y},\widehat{z},\widehat{k})\big |^2dt\Bigg \}\\
		&\le \mathbb{E}\Bigg \{-\left \langle \widehat{\eta},\widehat{x}_T  \right \rangle +\left \langle \widehat{\xi }_0,\widehat{\vartheta}_0  \right \rangle+\int_{0}^{T}\Big [\left \langle \widehat{\varphi},\widehat{x}\right \rangle +\left \langle \widehat{\psi},\widehat{y}\right \rangle + \left \langle \widehat{\gamma },\widehat{z}  \right \rangle +\left \langle \widehat{\beta },\widehat{k}  \right \rangle \Big ]dt\Bigg \},
	\end{aligned}
\end{equation}

The following proofs will be divided into two cases according to Assumption \ref{ass:3.2}-(i).

\textbf{Case 1}: $\mu>0$ and $v=0$. By applying the domination conditions \eqref{eq:3.1} in Assumption \ref{ass:3.2}-(ii) to the estimate \eqref{eq:3.13}, we get
\begin{equation}\label{eq:3.17}
	\begin{aligned}
	\mathbb E\bigg[\displaystyle \sup_{t\in [0,T]} |\widehat{x}|^2\bigg ]&\le K\mathbb E\Bigg \{\int_{0}^{T}|\widehat{\psi } |^2dt+\int_{0}^{T}|\widehat{\gamma} |^2dt+\int_{0}^{T}\|\widehat{\beta} \|_{l^2(\mathbb{R}^n)}^2dt+\displaystyle \sup_{t\in [-\delta,0]} |\widehat{\xi } |^2\\&\quad+\int_{0}^{T}\big |Q(\widehat{y},\widehat{z},\widehat{k})\big |^2dt+\int_{0}^{T}\big |Q_{-}(\widehat{y},\widehat{z},\widehat{k})\big |^2dt\Bigg \}.
	\end{aligned}
\end{equation}
By the time-shifting transformation and the property that $\bar{B}=\bar{C}=\bar{D}_j=0$, when $t\in[0,\delta]$, we derive that
\begin{equation}\label{eq:6.19}
	\begin{aligned}
		&\quad\int_{0}^{T}\big |Q_{-}(\widehat{y},\widehat{z},\widehat{k})\big |^2dt\\&=\int_{0}^{\delta}\big |Q_{-}(\widehat{y},\widehat{z},\widehat{k})\big |^2dt+\int_{\delta}^{T}\big |Q_{-}(\widehat{y},\widehat{z},\widehat{k})\big |^2dt\\&=\int_{-\delta}^{0}\big |B\widehat{y}+C\widehat{z}+\sum_{j=1}^{\infty}D_{j}\widehat{k}^{(j)}\big |^2dt+\int_{0}^{T-\delta}\big |Q(\widehat{y},\widehat{z},\widehat{k})\big |^2dt
		\\&\le K\bigg\{\displaystyle \sup_{t\in[-\delta,0]}|\widehat{\vartheta} |^2+\int_{-\delta}^{0}|\widehat{\tau}|^2dt+\int_{-\delta}^{0}\|\widehat{\chi}\|_{l^2(\mathbb{R}^n)}^2dt+\int_{0}^{T}\big |Q(\widehat{y},\widehat{z},\widehat{k})\big |^2dt\biggr\}.
	\end{aligned}
\end{equation}
Then by substituting \eqref{eq:6.19} into \eqref{eq:3.17}, the estimate \eqref{eq:3.17} can be rewritten as follows:
\begin{equation}\label{eq:6.20}
 	\begin{aligned}
 		\mathbb E\bigg[\displaystyle \sup_{t\in [0,T]} |\widehat{x}|^2\bigg ]&\le K\mathbb E\Bigg \{\int_{0}^{T}|\widehat{\psi } |^2dt+\int_{0}^{T}|\widehat{\gamma} |^2dt+\int_{0}^{T}\|\widehat{\beta} \|_{l^2(\mathbb{R}^n)}^2dt+\int_{0}^{T}\big |Q(\widehat{y},\widehat{z},\widehat{k})\big |^2dt\\&\qquad
 +\displaystyle \sup_{t\in [-\delta,0]} |\widehat{\xi } |^2
 +\displaystyle \sup_{t\in[-\delta,0]}|\widehat{\vartheta} |^2+\int_{-\delta}^{0}|\widehat{\tau}|^2dt+\int_{-\delta}^{0}\|\widehat{\chi}\|_{l^2(\mathbb{R}^n)}^2dt\Bigg \}.
 	\end{aligned}
\end{equation}
Applying the Lipschitz condition and the time-shifting transformation to the estimate \eqref{eq:3.14} leads to
\begin{equation}\label{eq:6.22}
	\begin{aligned}
\mathbb E\bigg[\displaystyle \sup_{t\in [0,T]} |\widehat{y}|^2+\int_{0}^{T}|\widehat{z}|^2dt+\int_{0}^{T}\|\widehat{k}\|_{l^2(\mathbb{R}^n)}^2dt\bigg ]\le K\mathbb{E}\bigg \{|\widehat{\eta}|^2+\int_{0}^{T}|\widehat{\varphi}|^2dt+\displaystyle \sup_{t\in [0,T]} |\widehat{x} |^2+\displaystyle \sup_{t\in [-\delta,0]} |\widehat{\xi}|^2\bigg\}.
	\end{aligned}
\end{equation}
Hence, combining \eqref{eq:6.20} and \eqref{eq:6.22} yields
\begin{equation}\label{eq:3.19}
\mathbb E\bigg[\displaystyle \sup_{t\in [0,T]} |\widehat{x}|^2+\displaystyle \sup_{t\in [0,T]} |\widehat{y}|^2+\int_{0}^{T}|\widehat{z}|^2dt+\int_{0}^{T}\|\widehat{k}\|_{l^2(\mathbb{R}^n)}^2dt\bigg ]\le K\mathbb{E}\Bigg \{\widehat{\mathrm{J}}+\int_{0}^{T}\big |Q(\widehat{y},\widehat{z},\widehat{k})\big |^2dt\Bigg \},	
\end{equation}
where $\widehat{\mathrm{J}}$ is defined by \eqref{eq:3.12}. Finally, we continue to combine \eqref{eq:3.16} and \eqref{eq:3.19} and use the inequality $ab\le\frac{1}{4\varepsilon}a^2+\epsilon b^2$ to get
\begin{equation}\label{eq:3.20}
	\begin{aligned}
		&\quad\mathbb E\bigg[\displaystyle \sup_{t\in [0,T]} |\widehat{x}|^2+\displaystyle \sup_{t\in [0,T]} |\widehat{y}|^2+\int_{0}^{T}|\widehat{z}|^2dt+\int_{0}^{T}\|\widehat{k}\|_{l^2(\mathbb{R}^n)}^2dt\bigg ]\\
		&\le K\mathbb{E}\Bigg \{\widehat{\mathrm{J}}-\left \langle \widehat{\eta},\widehat{x}_T  \right \rangle +\left \langle \widehat{\xi }_0,\widehat{\vartheta}_0  \right \rangle+\int_{0}^{T}\Big [\left \langle \widehat{\varphi},\widehat{x}\right \rangle +\left \langle \widehat{\psi},\widehat{y}\right \rangle + \left \langle \widehat{\gamma },\widehat{z}  \right \rangle +\left \langle \widehat{\beta },\widehat{k}  \right \rangle \Big ]dt\Bigg \}\\
		&\le K\mathbb{E}\Bigg \{\widehat{\mathrm{J}}+|\widehat{\eta}|\Big (\displaystyle \sup_{t\in [0,T]} |\widehat{x}|\Big )+\Big (\displaystyle \sup_{t\in [-\delta,0]} |\widehat{\xi}|\Big )\Big (\displaystyle \sup_{t\in [-\delta,0]} |\widehat{\vartheta}|\Big )+\bigg (\int_{0}^{T}|\widehat{\varphi}|dt\bigg )\Big (\displaystyle \sup_{t\in [0,T]} |\widehat{x}|\Big )\\
		&\qquad +\bigg (\int_{0}^{T}|\widehat{\psi}|dt\bigg )\Big (\displaystyle \sup_{t\in [0,T]} |\widehat{y}|\Big )+\int_{0}^{T}|\widehat{\gamma}||\widehat{z}|dt+\int_{0}^{T}\|\widehat{\beta}\|_{l^2(\mathbb{R}^n)}\|\widehat{k}\|_{l^2(\mathbb{R}^n)}dt\Bigg\}\\
		&\le K\mathbb{E}\Bigg \{\widehat{\mathrm{J}}+2\varepsilon\bigg [ \displaystyle \sup_{t\in [0,T]} |\widehat{x}|^2+\displaystyle \sup_{t\in [0,T]} |\widehat{y}|^2+\int_{0}^{T}|\widehat{z}|^2dt+\int_{0}^{T}\|\widehat{k}\|_{l^2(\mathbb{R}^n)}^2dt\bigg ]\Bigg\}.
	\end{aligned}
\end{equation}

By taking $\varepsilon$ small enough such that $2K\varepsilon<1$, we can easily obtain the desired estimate \eqref{eq:3.11} and the proof in this case is finished.

\textbf{Case 2}: $\mu=0$ and $v>0$. Differently, we apply the Lipschitz conditions and the time-shifting transformation to the estimate \eqref{eq:3.13} to get
\begin{equation}\label{eq:6.26}
	\begin{aligned}
		\mathbb E\bigg[\displaystyle \sup_{t\in [ 0,T]} |\widehat{x}|^2\bigg ]\le& K\mathbb{E}\Bigg \{\displaystyle \sup_{t\in [0,T]} |\widehat{y}|^2+\int_{0}^{T}|\widehat{z}|^2dt+\int_{0}^{T}\|\widehat{k}\|_{l^2(\mathbb{R}^n)}^2dt+\displaystyle \sup_{t\in[-\delta,0]}|\widehat{\xi}|^2+\displaystyle \sup_{t\in[-\delta,0]}|\widehat{\vartheta} |^2+\int_{-\delta}^{0}|\widehat{\tau}|^2dt\\&+\int_{-\delta}^{0}\|\widehat{\chi}\|_{l^2(\mathbb{R}^n)}^2dt
		+\int_{0}^{T}|\widehat{\psi } |^2dt+\int_{0}^{T}|\widehat{\gamma} |^2dt+\int_{0}^{T}\|\widehat{\beta} \|_{l^2(\mathbb{R}^n)}^2dt\Bigg \}.
	\end{aligned}
\end{equation}
By the domination conditions \eqref{eq:3.1} in Assumption \ref{ass:3.2}-(ii), we deduce from \eqref{eq:3.14} and obtain
\begin{equation}\label{eq:3.22}
	\begin{aligned}
		&\quad\mathbb E\bigg[\displaystyle \sup_{t\in [0,T]} |\widehat{y}|^2+\int_{0}^{T}|\widehat{z}|^2dt+\int_{0}^{T}\|\widehat{k}\|_{l^2(\mathbb{R}^n)}^2dt\bigg ]\\
		&\le K\mathbb{E}\bigg \{|G\widehat{x}_T|^2+\int_{0}^{T}|P(\widehat{x})|^2dt+\int_{0}^{T}|P_{-}(\widehat{x})|^2dt+|\widehat{\eta} |^2+\int_{0}^{T}|\widehat{\varphi}|^2dt
		 \Bigg \}.
	\end{aligned}
\end{equation}
Using the the time-shifting transformation and the property that $\bar{A}=0$, when $t\in[0,\delta]$, we have
\begin{equation}\label{eq:6.28}
	\begin{aligned}
		\int_{0}^{T}|P_{-}(\widehat{x})|^2dt&=\int_{0}^{\delta}|P_{-}(\widehat{x})|^2dt+\int_{\delta}^{T}|P_{-}(\widehat{x})|^2dt=\int_{-\delta}^{0}|Ax|^2dt+\int_{0}^{T-\delta}|P(\widehat{x})|^2dt\\&\le K\bigg\{\displaystyle \sup_{t\in [-\delta,0]}|\widehat{\xi}|^2+\int_{0}^{T}|P(\widehat{x})|^2dt\bigg\}.
	\end{aligned}
\end{equation}
Then putting \eqref{eq:6.28} into \eqref{eq:3.22} yields
\begin{equation}\label{eq:6.29}
	\begin{aligned}
		&\quad\mathbb E\bigg[\displaystyle \sup_{t\in [0,T]} |\widehat{y}|^2+\int_{0}^{T}|\widehat{z}|^2dt+\int_{0}^{T}\|\widehat{k}\|_{l^2(\mathbb{R}^n)}^2dt\bigg ]\\
		&\le K\mathbb{E}\bigg \{|G\widehat{x}_T|^2+\int_{0}^{T}|P(\widehat{x})|^2dt+\displaystyle \sup_{t\in [-\delta,0]}|\widehat{\xi}|^2+|\widehat{\eta} |^2+\int_{0}^{T}|\widehat{\varphi}|^2dt
		\Bigg \}.
	\end{aligned}
\end{equation}
Thus, combining \eqref{eq:6.26} and \eqref{eq:6.29}, we can derive
\begin{equation}\label{eq:3.23}
	\begin{aligned}
		\mathbb E\bigg[\displaystyle \sup_{t\in [0,T]} |\widehat{x}|^2+\displaystyle \sup_{t\in [0,T]} |\widehat{y}|^2+\int_{0}^{T}|\widehat{z}|^2dt+\int_{0}^{T}\|\widehat{k}\|_{l^2(\mathbb{R}^n)}^2dt\bigg ]
		\le K\mathbb{E}\Bigg \{\widehat{\mathrm{J} }+|G\widehat{x}_T|^2+\int_{0}^{T}|P(\widehat{x})|^2dt\Bigg \}.
	\end{aligned}
\end{equation}
where $\widehat{\mathrm{J}}$ is defined by \eqref{eq:3.12}. At last, \eqref{eq:3.16} and \eqref{eq:3.23} work together to turn out
\begin{equation}
	\begin{aligned}
		&\quad\mathbb E\bigg[\displaystyle \sup_{t\in [0,T]} |\widehat{x}|^2+\displaystyle \sup_{t\in [0,T]} |\widehat{y}|^2+\int_{0}^{T}|\widehat{z}|^2dt+\int_{0}^{T}\|\widehat{k}\|_{l^2(\mathbb{R}^n)}^2dt\bigg ]\\
		&\le K\mathbb{E}\Bigg \{\widehat{\mathrm{J} }-\left \langle \widehat{\eta},\widehat{x}_T  \right \rangle +\left \langle \widehat{\xi }_0,\widehat{\vartheta}_0  \right \rangle+\int_{0}^{T}\Big [\left \langle \widehat{\varphi},\widehat{x}\right \rangle +\left \langle \widehat{\psi},\widehat{y}\right \rangle + \left \langle \widehat{\gamma },\widehat{z}  \right \rangle +\left \langle \widehat{\beta },\widehat{k}  \right \rangle \Big ]dt\Bigg \}.
	\end{aligned}
\end{equation}
The remaining proof is the same as \eqref{eq:3.20} in Case 1 and then we finish the proof in this case. Consequently, the whole proof of the lemma is completed.
\end{proof}

Next, we prove a continuation lemma based on the priori estimate in Lemma \ref{lem:3.2}.
\begin{lem}\label{lem:3.3}
Let Assumption \ref{ass:3.1} and Assumption \ref{ass:3.2} be satisfied. If for some $\alpha_{0}\in[0,1)$, FBSDELDA \eqref{eq:3.8} admits a unique solution in $N_{\mathbb{F} }^2(0,T;\mathbb{R}^{n(2+d)}\times l^2(\mathbb{R}^n) )$ for any $(\pi(\cdot) ,\eta ,\rho (\cdot ))\in \mathcal{H} [-\delta,T]$, then there exists an absolute constant $\delta_0>0$ such that the same conclusion is also true for $\alpha=\alpha_{0}+\delta$ with $\delta\in(0,\delta_0]$ and $\alpha\le1$.
\end{lem}
\begin{proof}
Let $\delta_0>0$ be determined below. For any $\theta (\cdot )\in N_{\mathbb{F} }^2(0 ,T;\mathbb{R}^{n(2+d)}\times l^2(\mathbb{R}^n) )$, we introduce the following FBSDELDA with unknow $\Theta (\cdot):=(X(\cdot)^\top,Y(\cdot)^\top,Z(\cdot)^\top,K(\cdot)^\top)\in N_{\mathbb{F}}^2(0,T;\mathbb{R}^{n(2+d)}\times l^2(\mathbb{R}^n) )$ where $K(\cdot):=(K^{(1)}(\cdot)^\top,K^{(2)}(\cdot)^\top,\cdots )^\top$ (compared to \eqref{eq:3.8} with $\alpha$ = $\alpha_{0}+\delta$):
\begin{equation}\label{eq:3.25}
	\left\{\begin{aligned}
		dX(t)  =&\Big\{-(1-\alpha_0)\mu\Big[B(t)^{\top}Q(t,Y_t,Z_t,K_t)+\bar{B}(t)^{\top}Q_{-}(t,Y_t,Z_t,K_t)\Big]\\&~~+\alpha_0 b(t,\Theta(t), \Theta_{-}(t),Y_{+}(t),Z_{+}(t),K_{+}(t))+\widetilde{\psi}(t)\Big\}dt\\
		&+\Big\{-(1-\alpha_0)\mu\Big[C(t)^{\top}Q(t,Y_t,Z_t,K_t)+\bar{C}(t)^{\top}Q_{-}(t,Y_t,Z_t,K_t)\Big]\\&~~+\alpha_0 \sigma(t,\Theta(t), \Theta_{-}(t),Y_{+}(t),Z_{+}(t),K_{+}(t))+\widetilde{\gamma}(t)\Big\}d W(t)\\&+\sum_{i=1}^{\infty }\Big\{-(1-\alpha_0)\mu\Big[D_i(t)^{\top}Q(t,Y_{t-},Z_t,K_t)+\bar{D}_i(t)^{\top}Q_{-}(t,Y_{t-},Z_t,K_t)\Big]\\&~~+\alpha_0 g^{(i)}(t,\Theta(t-), \Theta_{-}(t-),Y_{+}(t-),Z_{+}(t),K_{+}(t))+\widetilde{\beta}^{(i)}(t)\Big\}dH^{(i)}(t),\quad t \in[0, T], \\
		dY(t) =&\Big\{-(1-\alpha_0)v\Big[A(t)^{\top}P(t,X_t)+\bar{A}(t)^{\top}P_{-}(t,X_t)\Big]\\&~~+\alpha_0f(t, \Theta(t), X_{-}(t),\Theta_{+}(t))+\widetilde{\varphi}(t)\Big\}dt\\&+Z(t)d W(t)+\sum_{i=1}^{\infty }K^{(i)}(t)dH^{(i)}(t), \quad t \in[0, T],\\
		X(t)=&\lambda(t),\quad Y(t)=\mu(t),\quad Z(t)=\rho(t),\quad K(t)=\varsigma(t),\quad t \in[-\delta , 0],\\
		Y(T)=&\alpha_{0}\Phi(X(T))+(1-\alpha_{0})vG^{\top}GX(T)+\widetilde{\eta},\\
		X(t)=&Y(t)=Z(t)=K(t)=0,\quad t\in(T,T+\delta],
	\end{aligned}\right.
\end{equation}
where
 \begin{equation}
	\left\{\begin{aligned}
 		&\widetilde{\psi}(t)=\psi(t) +\delta b(t,\theta (t),\theta_{-}(t ),y_{+}(t),z_{+}(t),k_{+}(t))+\delta \mu  \Big [B(t)^\top Q(t,y_t,z_t,k_t)+\bar{B}(t)^\top Q_{-}(t,y_t,z_t,k_t)\Big ],\\
		&\widetilde{\gamma}(t)=\gamma(t) +\delta \sigma(t,\theta (t),\theta_{-}(t ),y_{+}(t),z_{+}(t),k_{+}(t))+\delta \mu  \Big [C(t)^\top Q(t,y_t,z_t,k_t)+\bar{C}(t)^\top Q_{-}(t,y_t,z_t,k_t)\Big ],\\
		&\widetilde{\beta}^{(i)}(t)=\beta^{(i)}(t) +\delta g^{(i)}(t,\theta (t),\theta_{-}(t ),y_{+}(t),z_{+}(t),k_{+}(t))+\delta \mu  \Big [D_i(t)^\top Q(t,y_t,z_t,k_t)+\bar{D}_i(t)^\top Q_{-}(t,y_t,z_t,k_t)\Big ],\\
		&\widetilde{\varphi}(t)=\varphi(t)+\delta f(t,\theta (t),x_{-}(t),\theta_{+}(t))+\delta v\Big[A(t)^\top P(t,x_t)+\bar{A}(t)^\top P_{-}(t,x_t)\Big],\\
		&\widetilde{\eta}=\eta +\delta \Phi(x(T))-\delta vG^\top Gx(T).
	\end{aligned}\right.
\end{equation}
Similar to before, for $\Theta(t-)$ and $\Theta_{-}(t-)$, $t-$ only works on $X$, $X_{-}$ and $Y$, $Y_{-}$. Furthermore, we also denote $\Lambda(\cdot)=(\lambda(\cdot),\mu(\cdot),\rho(\cdot),\varsigma(\cdot))$ and  $\widetilde{\rho} (\cdot )=(\widetilde{\varphi} (\cdot )^\top ,\widetilde{\psi} (\cdot )^\top,\widetilde{\gamma} (\cdot )^\top,\widetilde{\beta} (\cdot )^\top)^\top$. Then it is easy to check that $(\Lambda,\widetilde{\eta},\widetilde{\rho})\in\mathcal{H}[-\delta,T]$. By our assumptions, the FBSDELDA \eqref{eq:3.25} admits a unique solution $\Theta(\cdot )\in N_{\mathbb{F} }^2(0,T;\mathbb{R}^{n(2+d)}\times l^2(\mathbb{R}^n) )$. In fact, we have established a mapping
\begin{equation}
\Theta (\cdot)=\mathcal{T}_{\alpha _0+\delta }\big (\theta (\cdot )\big ): N_{\mathbb{F} }^2(0,T;\mathbb{R}^{n(2+d)}\times l^2(\mathbb{R}^n) )\to N_{\mathbb{F} }^2(0 ,T;\mathbb{R}^{n(2+d)}\times l^2(\mathbb{R}^n) ).\nonumber
\end{equation}
In the following, we shall prove that the above mapping is contractive when $\delta$ is small enough.

Let $\theta(\cdot ),\bar{\theta}(\cdot )\in N_{\mathbb{F} }^2(0  ,T;\mathbb{R}^{n(2+d)}\times l^2(\mathbb{R}^n) )$ and $\Theta (\cdot)=\mathcal{T}_{\alpha _0+\delta }\big (\theta (\cdot )\big )$, $\bar{\Theta} (\cdot)=\mathcal{T}_{\alpha _0+\delta }\big (\bar{\theta} (\cdot )\big )$. Similarly, we denote $\widehat{\theta }(\cdot)=\theta(\cdot) -\bar{\theta }(\cdot)$, $\widehat{\Theta }(\cdot)=\Theta(\cdot) -\bar{\Theta }(\cdot)$, etc. By applying Lemma \ref{lem:3.2}, we have (the argument $t$ is suppressed for simplicity)
\begin{equation}
	\begin{aligned}
		&\quad\|\widehat{\Theta } \|^2_{N_{\mathbb{F} }^2(0,T  ;\mathbb{R}^{n(2+d)}\times l^2(\mathbb{R}^n) )}\\&=\mathbb{E}\bigg [\displaystyle\sup_{t\in[0 ,T]}|\widehat{X}|^2+\displaystyle\sup_{t\in [0,T ]}|\widehat{Y}|^2+\int_{0}^{T}|\widehat{Z}|^2dt+\int_{0}^{T }\|\widehat{K}\|_{l^2(\mathbb{R}^n)}^2dt \bigg ]\\
		&\le K\delta ^2\mathbb{E}\Bigg \{\Big |\big (\Phi (x_T)-\Phi (\bar{x}_T )\big )-v G^\top G\widehat{x}_T\Big |^2\\
		&\quad +\int_{0}^{T}\bigg |\big (b(\theta,\theta_{-},y_{+},z_{+},k_{+})-b(\bar{ \theta },\bar{\theta} _{-},\bar{y} _{+},\bar{z} _{+},\bar{k} _{+})\big )+\mu \big[B^\top Q(\widehat{y},\widehat{z},\widehat{k})+\bar{B}^\top Q_{-}(\widehat{y},\widehat{z},\widehat{k})\big]\bigg |^2dt\\
		&\quad +\int_{0}^{T}\bigg |\big (\sigma(\theta,\theta_{-},y_{+},z_{+},k_{+})-\sigma(\bar{ \theta },\bar{\theta} _{-},\bar{y} _{+},\bar{z} _{+},\bar{k} _{+})\big )+\mu \big[C^\top Q(\widehat{y},\widehat{z},\widehat{k})+\bar{C}^\top Q_{-}(\widehat{y},\widehat{z},\widehat{k})\big]\bigg |^2dt\\
		&\quad +\int_{0}^{T}\sum_{i=1}^{\infty } \bigg \|\big (g^{(i)}(\theta,\theta_{-},y_{+},z_{+},k_{+})-g^{(i)}(\bar{ \theta },\bar{\theta} _{-},\bar{y} _{+},\bar{z} _{+},\bar{k} _{+})\big )+\mu \big[D_i^\top Q(\widehat{y},\widehat{z},\widehat{k})+\bar{D}_i^\top Q_{-}(\widehat{y},\widehat{z},\widehat{k})\big]\bigg \|_{l^2(\mathbb{R}^n)}^2dt\\
		&\quad+\int_{0}^{T}\bigg|\big (f(\theta,x_{-},\theta_{+})-f(\bar{ \theta },\bar{x} _{-},\bar{\theta} _{+})\big )+v\big[A^\top P(\widehat{x})+\bar{A}^\top P_{-}(\widehat{x})\big] \bigg|
		\Bigg\}.\nonumber
	\end{aligned}
\end{equation}
Due to the Lipschitz continuity of $(\Phi,\Gamma)$ and the boundedness of $G$, $A(\cdot)$, $\bar{A}(\cdot)$, $B(\cdot)$, $\bar{B}(\cdot)$, $C(\cdot)$, $\bar{C}(\cdot)$, $D_i(\cdot)$, $\bar{D}_i(\cdot)$, there exists a new constant $K'>0$ independent of $\alpha_{0}$ and $\delta$ such that
\begin{equation}
 \|\widehat{\Theta } \|^2_{N_{\mathbb{F} }^2(0,T;\mathbb{R}^{n(2+d)}\times l^2(\mathbb{R}^n) )}\le K'\delta ^2 \|\widehat{\theta } \|^2_{N_{\mathbb{F} }^2(0,T;\mathbb{R}^{n(2+d)}\times l^2(\mathbb{R}^n) )}.\nonumber
\end{equation}
By selecting $\delta_0=1/(2\sqrt{K'})$, when $\alpha\in(0,\delta_0]$, we can find that the mapping $\mathcal{T}_{\alpha _0+\delta }$ is contractive. Thus, the mapping $\mathcal{T}_{\alpha _0+\delta }$ admits a unique fixed point which is just the unique solution to FBSDELDA \eqref{eq:3.8}. The proof is finished.
\end{proof}

We summarize the above analysis to give the following proof.
\begin{proof}[proof of Theorem \ref{thm:3.1}]
Firstly, the unique solvability of FBSDELDA \eqref{eq:1.1} in the space $N_{\mathbb{F} }^2(0,T;\mathbb{R}^{n(2+d)}\times l^2(\mathbb{R}^n) )$ is deduced from the unique solvability of FBSDELDA \eqref{eq:3.10} and Lemma \ref{lem:3.3}. Secondly, in Lemma \ref{lem:3.2}, by taking $\alpha=1$, $\big(\pi(\cdot),\eta,\rho(\cdot)\big)=(0,0,0)$, and$\big(\bar{\pi}(\cdot),\bar{\eta},\bar{\rho}(\cdot)\big)=\Big(\bar{\Lambda}(\cdot)-\Lambda(\cdot), \bar{\Phi}(\bar{x}_T)-\Phi(\bar{x}_T), \bar{\Gamma}\big(\cdot,\bar{\theta}(\cdot),\bar{\theta}_{-}(\cdot),\bar{\theta}_{+}(\cdot)\big)-\Gamma\big(\cdot,\bar{\theta}(\cdot),\bar{\theta}_{-}(\cdot),\bar{\theta}_{+}(\cdot)\big)\Big)$, we get the estimate \eqref{eq:3.6} in Theorem \ref{thm:3.1} from the estimate \eqref{eq:3.11} in Lemma \ref{lem:3.2}. Finally, by selecting the coefficients $\big(\bar{\Lambda},\bar{\Phi},\bar{\Gamma}\big)=(0,0,0)$, we get \eqref{eq:3.4} from \eqref{eq:3.6}. Then the proof is completed.
\end{proof}

\section{Applications in linear quadratic problem}\label{sec:4}
In this section, we will study two kinds of linear quadratic optimal control problems and then we find that the Hamiltonian systems arising from these linear quadratic problems are FBSDELDAs satisfying the domination-monotonicity conditions mentioned in Section \ref{sec:3}. Therefore, we have the conclusion that they are uniquely solvable by virtue of Theorem \ref{thm:3.1}. Actually, exploring the solvability of these Hamiltonian systems is also one of our motivations in this paper. It should be noted that we assume that the Brownian motion is 1-dimensional in this section.
\subsection{Forward LQ stochastic control problem}\label{sec:4.1}
Firstly, we give the following control system driven by a linear forward SDEDL:
 \begin{equation}\label{eq:4.1}
 	\left\{\begin{aligned}
 		dx_t=&\big(A_tx_t+\bar{A}_tx_{t-\delta }+B_tv_t+\bar{B}_tv_{t-\delta }\big)dt+\big(C_tx_t+\bar{C}_tx_{t-\delta }+D_tv_t+\bar{D}_tv_{t-\delta }\big)dW_t\\
 		&+\sum_{i=1}^{\infty }\big(E_t^{(i)}x_{t-}+\bar{E}_t^{(i)}x_{(t-\delta)- }+F_t^{(i)}v_t+\bar{F}^{(i)}_tv_{t-\delta }\big)dH_t^{(i)},\quad t\in[0,T],\\
 		x_0=&a, \quad x_t=\lambda_t,\quad v_t=\zeta _t,\quad t\in [-\delta ,0),
 	\end{aligned}\right.
 \end{equation}
where $\delta>0$ is a  constant time delay, $a\in \mathbb{R}^n$, $\lambda_t\in C(-\delta,0;\mathbb{R}^{n} )$, $\zeta_t\in C(-\delta,0;\mathbb{R}^{m} )$. Moreover, $A_t, C_t, E^{(i)}_t, \bar{A}_t, \bar{C}_t, \bar{E}^{(i)}_t\in L^{\infty } _{\mathbb{F}}(0,T;\mathbb{R}^{n\times n} )$, $B_t, D_t, F^{(i)}_t\in L^{\infty } _{\mathbb{F}}(0,T;\mathbb{R}^{n\times m} )$, $\bar{B}_t, \bar{D}_t, \bar{F}^{(i)}_t\in L^{\infty } _{\mathbb{G}}(0,T;\mathbb{R}^{n\times m} )$, where $\bar{B}_t=\bar{D}_t=\bar{F}_t=0$, when $t\in[0,\delta]$. The admissible control set is denoted by $\mathcal{V}_{ad}$, in which each element $v(\cdot )\in \mathcal{V}_{ad}$ has the following form
\begin{equation}
	\left\{\begin{aligned}
		&v_t=\zeta _t,\quad t \in [-\delta ,0),\\
		&v_t=v_t\in  L^{2} _{\mathbb{F}}(0,T;\mathbb{R}^{m} ),\quad t\in [0,T],\nonumber
	\end{aligned}\right.
\end{equation}
which is called an admissible control. By Lemma \ref{lem:2.2}, we know that SDEDL \eqref{eq:4.1} admits a unique solution $x(\cdot)\in \mathcal{S}^{2} _{\mathbb{F}}(0,T;\mathbb{R}^{n} )$.

Next, we continue to give a cost functional in quadratic form as follows:
\begin{equation}\label{eq:4.2}
	\begin{aligned}
	J\big(v(\cdot )\big)=\frac{1}{2} \mathbb{E}\bigg \{\int_{0}^{T}& \Big ( \left \langle Q_tx_t,x_t \right \rangle + \left \langle \bar{Q}_tx_{t-\delta},x_{t-\delta } \right \rangle +\left \langle R_tv_t,v_t \right \rangle+\left \langle \bar{R}_tv_{t-\delta},v_{t-\delta } \right \rangle\\&+2\left \langle S_tx_t,v_t \right \rangle+2\left \langle \bar{S}_{t}x_{t-\delta},v_{t-\delta} \right \rangle\Big )dt+\left \langle Gx_T,x_T \right \rangle \bigg \},
	\end{aligned}
\end{equation}
where $Q_t,\bar{Q}_t \in L^{\infty }_{\mathbb{F} }(0,T;\mathbb{S}^n )$, $R_t,\bar{R}_t \in L^{\infty }_{\mathbb{F} }(0,T;\mathbb{S}^m)$, $S_t,\bar{S}_t\in L^{\infty }_{\mathbb{F} }(0,T;\mathbb{R}^{n\times m} )$ and $G$ is an $\mathcal{F}$-measurable $n\times n$ symmetric matrix-valued bounded random variable. In addition,  $\bar{Q}_t=\bar{R}_t=\bar{S}_t=0$, when $t\in (T,T+\delta]$.

Now, we propose our main problem as follows:\\
\textbf{Problem(LQDL).}The problem is to find an admissible control $u(\cdot)\in \mathcal{V}_{ad}$ such that
\begin{equation}
	J\big(u(\cdot )\big)=\underset{v(\cdot )\in \mathcal{V}_{ad} }{\textrm{inf}}J\big(v(\cdot )\big) .
\end{equation}
Then such an admissible control $u(\cdot)$ is called an optimal control, and $x^{u}(\cdot)$ is called the corresponding optimal trajectory.

 Moreover, we impose the following assumption:
 \begin{ass}\label{ass:4.1}
 (i)$G$ is nonnegative definite;\\
  (ii)For any $(\omega ,t)\in \Omega \times [0,T]$, $Q_t+\mathbb{E}^{\mathcal{F}_t}[\bar{Q}_{t+\delta}]$ is nonnegative definite;\\
  (iii)For any $(\omega ,t)\in \Omega \times [0,T]$, $R_t+\mathbb{E}^{\mathcal{F}_t}[\bar{R}_{t+\delta}]$ is positive definite. Besides, there exists a constant $\delta>0$ such that
 \begin{equation}
 	\left \langle (R_t+\mathbb{E}^{\mathcal{F}_t }[R_{t+\delta }])v,v  \right \rangle \ge \delta |v|^2,\quad a.s.\nonumber
 \end{equation}
 for any $v\in \mathbb{R}^m$ and for almost every $t\in [-\delta,T]$;\\
 (iv)$(Q_t+\mathbb{E}^{\mathcal{F}_t}[\bar{Q}_{t+\delta}])-(S_t+\mathbb{E}^{\mathcal{F}_t}[\bar{S}_{t+\delta}])^{\top}(R_t+\mathbb{E}^{\mathcal{F}_t}[\bar{R}_{t+\delta}])^{-1}(S_t+\mathbb{E}^{\mathcal{F}_t}[\bar{S}_{t+\delta}])$ is nonnegative definite.
 \end{ass}

\begin{rmk}
	This type of LQ optimal control problem for the system with delay and L\'{e}vy processes has been studied by Li and Wu \cite{li2016stochastic}. However, the cost functional in this paper is more complex and general, where $\left \langle \bar{Q}_tx_{t-\delta},x_{t-\delta } \right \rangle$, $\left \langle S_tx_t,v_t \right \rangle$ and $\left \langle \bar{S}_{t}x_{t-\delta},v_{t-\delta} \right \rangle$ are also considered.
\end{rmk}

Firstly, based on the result in Sun et al. \cite[Lemma 2.3]{sun2016open}, we introduce the following lemma for later use.
\begin{lem}\label{lem:4.1}
For any $v(\cdot)\in \mathcal{V}_{ad}$, let $x^v(\cdot)$ be the solution of the following equation:
 \begin{equation}\nonumber
	\left\{\begin{aligned}
		dx_t=&\big(A_tx_t+\bar{A}_tx_{t-\delta }+B_tv_t+\bar{B}_tv_{t-\delta }\big)dt+\big(C_tx_t+\bar{C}_tx_{t-\delta }+D_tv_t+\bar{D}_tv_{t-\delta }\big)dW_t\\
		&+\sum_{i=1}^{\infty }\big(E_t^{(i)}x_{t-}+\bar{E}_t^{(i)}x_{(t-\delta)- }+F_t^{(i)}v_t+\bar{F}^{(i)}_tv_{t-\delta }\big)dH_t^{(i)},\quad t\in[0,T],\\
		x_0=&0, \quad x_t=0,\quad v_t=0,\quad t\in [-\delta ,0),
	\end{aligned}\right.
\end{equation}
Then for any $\Theta(\cdot)\in L^2(0,T;\mathbb{R}^{m\times n})$, there exists a constant $\gamma>0$ such that
\begin{equation}\nonumber
	\mathbb{E}\int_{0}^{T}|v_t-\Theta_{t}x^v_t|^2dt\ge\gamma\mathbb{E}\int_{0}^{T}|v_t|^2dt,\qquad \forall v(\cdot)\in\mathcal{V}_{ad}.
\end{equation}
\end{lem}
\begin{proof}
We first define a bounded linear operator $\mathfrak{D} :\mathcal{V}_{ad}\to\mathcal{V}_{ad}$ by
\begin{equation}\nonumber
	\mathfrak{D}v=v-\Theta x^v.
\end{equation}
Then $\mathfrak{D}$ is bijective and its inverse $\mathfrak{D}^{-1}$ is given by
\begin{equation}\nonumber
	\mathfrak{D}^{-1}v=v+\Theta \widetilde{x}^v,
\end{equation}
where $\widetilde{x}^v(\cdot)$ is the solution of the following equation:
 \begin{equation}\nonumber
	\left\{\begin{aligned}
		d\widetilde{x}^v_t=&\big[(A_t+B_t\Theta_t)\widetilde{x}^v_t+(\bar{A}_t+\bar{B}_t\Theta_{t-\delta})\widetilde{x}^v_{t-\delta }+B_tv_t+\bar{B}_tv_{t-\delta }\big]dt\\&+\big[(C_t+D_t\Theta_t)\widetilde{x}^v_t+(\bar{C}_t+\bar{D}_t\Theta_{t-\delta})\widetilde{x}^v_{t-\delta }+D_tv_t+\bar{D}_tv_{t-\delta }\big]dW_t\\
		&+\sum_{i=1}^{\infty }\big((E^{(i)}_t+F^{(i)}_t\Theta_t)\widetilde{x}^v_{t-}+(\bar{E}^{(i)}_t+\bar{F}^{(i)}_t\Theta_{t-\delta})\widetilde{x}^v_{(t-\delta)- }+F_t^{(i)}v_t+\bar{F}^{(i)}_tv_{t-\delta }\big)dH_t^{(i)},\quad t\in[0,T],\\
		\widetilde{x}^v_0=&0, \quad \widetilde{x}^v_t=0,\quad v_t=0,\quad t\in [-\delta ,0).
	\end{aligned}\right.
\end{equation}
Noticing the bounded inverse theorem, we derive that $\mathfrak{D}^{-1}$ is bounded with norm $\|\mathfrak{D}^{-1}\|>0$. Based on this, we have
\begin{equation}\nonumber
	\mathbb{E}\int_{0}^{T}|v_t|^2dt=\mathbb{E}\int_{0}^{T}|(\mathfrak{D}^{-1}\mathfrak{D}v)_t|^2dt\le\|\mathfrak{D}^{-1}\|\mathbb{E}\int_{0}^{T}|(\mathfrak{D}v)_t|^2dt=\|\mathfrak{D}^{-1}\|\mathbb{E}\int_{0}^{T}|v_t-\Theta_tx^v_t|^2dt.
\end{equation}	
Finally, by taking $\gamma=\|\mathfrak{D}^{-1}\|^{-1}$, we finish the proof.
\end{proof}
Next, we will give the main result of this section. First of all, by the result of maximum principle for stochastic control system with delay and L\'{e}vy processes in Li and Wu \cite{li2014maximum} and \cite{li2016stochastic}, we can deduce the stochastic Hamiltonian system of SDEDL \eqref{eq:4.1} as follows:
\begin{equation}\label{eq:4.4}
	\left\{\begin{aligned}
		dx_t=&\big(A_tx_t+\bar{A}_tx_{t-\delta }+B_tu_t+\bar{B}_tu_{t-\delta }\big)dt+\big(C_tx_t+\bar{C}_tx_{t-\delta }+D_tu_t+\bar{D}_tu_{t-\delta }\big)dW_t\\
		&+\sum_{i=1}^{\infty }\big(E_t^{(i)}x_{t-}+\bar{E}_t^{(i)}x_{(t-\delta)- }+F_t^{(i)}v_t+\bar{F}^{(i)}_tv_{t-\delta }\big)dH_t^{(i)},\quad t\in[0,T],\\
		dy_t=&-\Big (A_t^\top y_t+C_t^\top z_t+\sum_{i=1}^{\infty } E_t^{(i)\top} k^{(i)}_t+Q_t x_t+S_t^\top u_t+\mathbb{E}^{\mathcal{F}_t}\big[\bar{A}_{t+\delta }^\top y_{t+\delta}+\bar{C}_{t+\delta }^\top z_{t+\delta}
		\\&~~~~~~+\sum_{i=1}^{\infty } \bar{E}_{t+\delta }^{(i)\top} k^{(i)}_{t+\delta}+\bar{Q}_{t+\delta }x_t+\bar{S}^\top_{t+\delta }u_t\big]\Big)dt+z_tdW_t+\sum_{i=1}^{\infty }k_t^{(i)}dH_t^{(i)},\quad t\in [0,T],\\
		x_0=&a, \quad x_t=\lambda_t,\quad u_t=\zeta _t,\quad t\in [-\delta ,0),\\
		y_T=&Gx_T,\quad y_t=z_t=k_t=0,\quad t\in [-\delta,0)\cup(T,T+\delta ],\\
		\Big(R_t+&\mathbb{E}^{\mathcal{F}_t}[\bar{R}_{t+\delta }]\Big)u_t+\Big (B_t^\top y_t+D_t^\top z_t+\sum_{i=1}^{\infty} F_t^{(i)\top} k^{(i)}_t+S_tx_t\\&+\mathbb{E}^{\mathcal{F}_t}\big[\bar{B}^\top_{t+\delta }y_{t+\delta }+\bar{D}^\top_{t+\delta }z_{t+\delta }+\sum_{i=1}^{\infty} \bar{F}_{t+\delta }^{(i)\top} k^{(i)}_{t+\delta }+\bar{S}_{t+\delta}x_t\big]\Big)=0.
	\end{aligned}\right.
\end{equation}

It is clear that the Hamiltonian system \eqref{eq:4.4} is described by a FBSDELDA. Therefore, with the help of Theorem \ref{thm:3.1}, we can get the following theorem.
\begin{thm}\label{thm:4.1}
	Under Assumption \ref{ass:4.1}, the above Hamiltonian system \eqref{eq:4.4} admits a unique solution $(\theta(\cdot),u(\cdot))$. Moreover, $u(\cdot)$ is the unique optimal control of Problem(LQDL) and $x(\cdot)$ is its corresponding optimal trajectory.  	
\end{thm}
\begin{proof}
Firstly, for simplicity, we denote by $\widetilde{R}_t=R_t+\mathbb{E}^{\mathcal{F}_t}[\bar{R}_{t+\delta }]$, $\widetilde{Q}_t=Q_t+\mathbb{E}^{\mathcal{F}_t}[\bar{Q}_{t+\delta }]$, $\widetilde{S}_t=S_t+\mathbb{E}^{\mathcal{F}_t}[\bar{S}_{t+\delta }]$.
 It is easy to find that $\widetilde{R}_t$ is  invertible, so we can solve $u(\cdot)$ from the last equation of Hamiltonian system \eqref{eq:4.4}.
\begin{equation}\label{eq:4.5}
	\begin{aligned}
	u_t=-&\widetilde{R}_t^{-1} \Big (B_t^\top y_t+D_t^\top z_t+\sum_{i=1}^{\infty} F_t^{(i)\top} k^{(i)}_t+\widetilde{S}_tx_t+\mathbb{E}^{\mathcal{F}_t}\big[\bar{B}^\top_{t+\delta }y_{t+\delta }+\bar{D}^\top_{t+\delta }z_{t+\delta }+\sum_{i=1}^{\infty} \bar{F}_{t+\delta }^{(i)\top} k^{(i)}_{t+\delta }\big]\Big).
	\end{aligned}
\end{equation}
Putting \eqref{eq:4.5} into the Hamiltonian system \eqref{eq:4.4}, we can get a FBSDELDA. Then we can easily verify that the coefficients of this FBSDELDA satisfy Assumption \ref{ass:3.1}, Assumption \ref{ass:3.2}-(i)-Case 1, and Assumption \ref{ass:3.2}-(ii). As for Assumption \ref{ass:3.2}-(iii), we give a detailed verification as follows:
\begin{equation}
	\begin{aligned}
	&\int_{0}^{T}\Big(\left \langle\Gamma (\theta ,\theta_{-},\theta_{+})-\Gamma (\bar{ \theta} ,\bar{ \theta }_{-},\bar{\theta}_{+}),\widehat{\theta}\right \rangle\Big)dt\\=&\int_{0}^{T}\Big(\left\langle A\widehat{x}+\bar{A}\widehat{x}_{-}+B\widehat{u}+\bar{B}\widehat{u}_{-},\widehat{y}\right\rangle+\left\langle C\widehat{x}+\bar{C}\widehat{x}_{-}+D\widehat{u}+\bar{D}\widehat{u}_{-},\widehat{z}\right\rangle+\left\langle E^{(i)}\widehat{x}+\bar{E}^{(i)}\widehat{x}_{-}+F^{(i)}\widehat{u}+\bar{F}^{(i)}\widehat{u}_{-},\widehat{k}^{(i)}\right\rangle
\\&-\left\langle A^\top\widehat{y}+C^\top\widehat{z}+\sum_{i=1}^{\infty}E^{(i)\top}\widehat{k}^{(i)}+Q\widehat{x}+S^\top\widehat{u}
+\mathbb{E}^{\mathcal{F}_t}\big[\bar{A}_{+}^\top\widehat{y}_{+}+\bar{C}_{+}^\top\widehat{z}_{+}
+\sum_{i=1}^{\infty}\bar{E}^{(i)\top}_{+}\widehat{k}^{(i)}_{+}+\bar{Q}_{+}\widehat{x}+\bar{S}_{+}^\top\widehat{u}\big],\widehat{x}\right\rangle\Big)dt\\
	=&\int_{0}^{T}\Big(\left\langle\widehat{u},B^\top\widehat{y}+D^\top\widehat{z}+\sum_{i=1}^{\infty}F^{(i)\top}\widehat{k}^{(i)}
-\widetilde{S}\widehat{x}+\mathbb{E}^{\mathcal{F}_t}\big[\bar{B}^\top_{+}\widehat{y}_{+}+\bar{D}^\top_{+}\widehat{z}_{+}
+\sum_{i=1}^{\infty}\bar{F}^{(i)\top}_{+}\widehat{k}^{(i)}_{+}\big]\right\rangle-\left\langle \widetilde{Q}\widehat{x},\widehat{x}\right\rangle\Big)dt\\
	=&\int_{0}^{T}\Big(-\left\langle\widetilde{R}^{-1}(Q(\widehat{y},\widehat{z},\widehat{k})+\widetilde{S}\widehat{x}),Q(\widehat{y},\widehat{z},\widehat{k})-\widetilde{S}\widehat{x}\right\rangle-\left\langle \widetilde{Q}\widehat{x},\widehat{x}\right\rangle\Big)dt\\
	=&\int_{0}^{T}\Big(-\left\langle\widetilde{R}^{-1}Q(\widehat{y},\widehat{z},\widehat{k}),Q(\widehat{y},\widehat{z},\widehat{k})\right\rangle-\left\langle(\widetilde{Q}-\widetilde{S}^{\top}\widetilde{R}^{-1}\widetilde{S})\widehat{x},\widehat{x}\right\rangle\Big)dt,
	\end{aligned}
\end{equation}
 where $Q(\widehat{y},\widehat{z},\widehat{k})$ is defined by \eqref{eq:6.4}. Then combining with Assummption \ref{ass:4.1}, we can finish the verification of Assumption \ref{ass:3.2}-(iii).
 Thus, by Theorem \ref{thm:3.1}, we know that this FBSDELDA admits a unique solution, which is equivalent to that the Hamiltonian system \eqref{eq:4.4} is uniquely solvable.

Next, we will prove the optimality of $u(\cdot)$ in the form of  \eqref{eq:4.5}. Firstly, let $v(\cdot)\in \mathcal{V}_{ad}$ be any other control and $x^{v}(\cdot)$ is the corresponding state. Then we shall explore the difference between $J\big(v(\cdot )\big)$ and $J\big(u(\cdot )\big)$:
	\begin{align}
		&\quad J\big(v(\cdot )\big)-J\big(u(\cdot )\big)\nonumber\\
		=&\frac{1}{2} \mathbb{E}\Bigg \{\int_{0}^{T} \Big ( \left \langle Q_tx^v_t,x^v_t \right \rangle -\left \langle Q_tx_t,x_t \right \rangle+ \left \langle \bar{Q}_tx^v_{t-\delta },x^v_{t-\delta } \right \rangle-\left \langle \bar{Q}_tx_{t-\delta },x_{t-\delta } \right \rangle  +\left \langle R_tv_t,v_t \right \rangle-\left \langle R_tu_t,u_t \right \rangle\nonumber\\
		&\qquad+\left \langle \bar{R}_tv_{t-\delta },v_{t-\delta } \right \rangle-\left \langle \bar{R}_tu_{t-\delta },u_{t-\delta } \right \rangle\Big )dt
		+\left \langle Gx^v_T,x^v_T \right \rangle-\left \langle Gx_T,x_T \right \rangle +2\left \langle S_tx^v_t,v_t \right \rangle-2\left \langle S_tx_t,u_t \right \rangle\nonumber\\&\qquad+2\left \langle \bar{S}_tx^v_{t-\delta},v_{t-\delta} \right \rangle-2\left \langle \bar{S}_tx_{t-\delta},u_{t-\delta} \right \rangle\Bigg \}\nonumber\\
		=&\frac{1}{2} \mathbb{E}\Bigg \{\int_{0}^{T} \Big (\left \langle Q_t(x^v_t-x_t),x^v_t-x_t \right \rangle +\left \langle \bar{Q}_t(x^v_{t-\delta }-x_{t-\delta }),x^v_{t-\delta }-x_{t-\delta } \right \rangle+\left \langle R_t(v_t-u_t), v_t-u_t\right \rangle\nonumber\\
		&\qquad+\left \langle \bar{R}_t(v_{t-\delta }-u_{t-\delta }),v_{t-\delta }-u_{t-\delta }  \right \rangle+2\left\langle S_t(x^v_t-x_t),v_t-u_t\right\rangle+2\left\langle \bar{S}_t(x^v_{t-\delta}-x_{t-\delta}),v_{t-\delta}-u_{t-\delta}\right\rangle\Big )dt\nonumber\\
		&\qquad+\left \langle G(x_T^v-x_T),x_T^v-x_T \right \rangle+ \int_{0}^{T}\Big (2\left \langle Q_tx_t,x_t^v-x_t \right \rangle  \nonumber+2\left \langle \bar{Q}_tx_{t-\delta },x^v_{t-\delta }-x_{t-\delta } \right \rangle+2\left \langle R_tu_t,v_t-u_t \right \rangle\nonumber\\
		&\qquad+2\left \langle \bar{R}_tu_{t-\delta },v_{t-\delta }-u_{t-\delta } \right \rangle +2\left \langle S_t(x^v_t-x_t),u_t \right \rangle+2\left \langle S_tx_t,v_t-u_t \right \rangle+2\left \langle \bar{S}_t(x^v_{t-\delta }-x_{t-\delta }),u_{t-\delta } \right \rangle\nonumber\\&\qquad+2\left \langle \bar{S}_tx_{t-\delta },v_{t-\delta }-u_{t-\delta } \right \rangle\Big )dt
		+ 2\left \langle Gx_T,x^v_T-x_T \right \rangle \Bigg \}.\nonumber
	\end{align}

Since $\bar{Q}=\bar{R}=\bar{S}=0$, when $t\in (T,T+\delta]$, by the time-shifting transformation and the initial conditions in \eqref{eq:4.4}, we have
\begin{equation}\label{eq:4.7}
	\begin{aligned}
		\quad &J\big(v(\cdot )\big)-J\big(u(\cdot )\big)\\
		=&\frac{1}{2} \mathbb{E}\Bigg \{\int_{0}^{T} \Big ( \left \langle \widetilde{Q}_t(x^v_t-x_t), x^v_t-x_t\right \rangle+ \left \langle \widetilde{R}_t(v_t-u_t),v_t-u_t \right \rangle \\&\qquad+2\left\langle\widetilde{S}_t(x^v_t-x_t),v_t-u_t\right\rangle+ \left \langle G(x_T^v-x_T),x_T^v-x_T \right \rangle  \Bigg \}+\Delta,
	\end{aligned}
\end{equation}
where
\begin{equation}
	\begin{aligned}
		\Delta =\mathbb{E}\Bigg \{\int_{0}^{T}&\Big (\left \langle \widetilde{Q}_tx_t , x^v_t-x_t \right \rangle +\left \langle \widetilde{R}_tu_t , v_t-u_t \right \rangle
		+\left\langle \widetilde{S}_t(x^v_t-x_t) ),u_t\right\rangle\\ &+\left\langle\widetilde{S}_tx_t,v_t-u_t\right\rangle \Big )dt+\left \langle Gx_T,x^v_T-x_T \right \rangle\Bigg \}.\nonumber
	\end{aligned}
\end{equation}
Applying It\^{o}'s formula to $\left \langle x^v_t-x_t,y_t \right \rangle $ leads to
\begin{equation}\label{eq:4.8}
	\begin{aligned}
		\mathbb{E}\Big [\left \langle x^v_T-x_T,y_T \right \rangle \Big ]
		=&\mathbb{E}\Bigg \{\int_{0}^{T}\bigg (\left \langle v_t-u_t,B_t^\top y_t+D_t^\top z_t+\sum_{i=1}^{\infty} F_t^{(i)\top} k^{(i)}_t+S_tx_t \right \rangle\\
		&\qquad+\left \langle v_{t-\delta }-u_{t-\delta },\bar{B}_t^\top y_t+\bar{D}_t^\top z_t+\sum_{i=1}^{\infty} \bar{F}_t^{(i)\top} k^{(i)}_t +\bar{S}_tx_{t-\delta}\right \rangle
		\\&\qquad-\left \langle \widetilde{Q}_tx_t,x_t^v-x_t \right \rangle-\left\langle\widetilde{S}_t(x^v_t-x_t),u_t\right\rangle\\&\qquad- \left\langle S_tx_t,v_t-u_t\right\rangle-\left\langle \bar{S}_tx_{t-\delta},v_{t-\delta}-u_{t-\delta}\right\rangle\bigg )dt\Bigg\}.
	\end{aligned}
\end{equation}
Moreover, due to the fact that $v_t=u_t$, $t\in [-\delta,0)$ and $y_t=z_t=k_t=0$, $t\in(T,T+\delta]$, \eqref{eq:4.8} can be rewritten as follows:
\begin{equation}
	\begin{aligned}
		&\mathbb{E}\Big [\left \langle x^v_T-x_T,Gx_T \right \rangle \Big ]
		\\=&\mathbb{E}\Bigg \{\int_{0}^{T}\bigg (\left\langle v_t-u_t,B_t^\top y_t+D_t^\top z_t+\sum_{i=1}^{\infty} F_t^{(i)\top} k^{(i)}_t+S_tx_t\right.  \\&\qquad+\left.\mathbb{E}^{\mathcal{F}_t}\big[\bar{B}^\top_{t+\delta }y_{t+\delta }+\bar{D}^\top_{t+\delta }z_{t+\delta }+\sum_{i=1}^{\infty} \bar{F}_{t+\delta }^{(i)\top} k^{(i)}_{t+\delta }+\bar{S}_{t+\delta}x_t\big]\right\rangle
		-\left \langle \widetilde{Q}_tx_t,x_t^v-x_t \right \rangle \\&\qquad-\left\langle\widetilde{S}_t(x^v_t-x_t),u_t\right\rangle- \left\langle\widetilde{S}_tx_t,v_t-u_t\right\rangle\bigg )dt\Bigg \}.
	\end{aligned}
\end{equation}
Consequently, combining with \eqref{eq:4.5}, we can infer that
$\Delta =0$. Thus, applying the nonnegative definiteness of $G$, we have
 	\begin{align}\label{eq:6.9}
 		\quad &J\big(v(\cdot )\big)-J\big(u(\cdot )\big)\nonumber\\
 		=&\frac{1}{2} \mathbb{E}\bigg \{\int_{0}^{T} \Big ( \left \langle \widetilde{Q}_t(x^v_t-x_t), x^v_t-x_t\right \rangle+ \left \langle \widetilde{R}_t(v_t-u_t),v_t-u_t \right \rangle \nonumber\\&\qquad+2\left\langle\widetilde{S}_t(x^v_t-x_t),v_t-u_t\right\rangle\Big)dt+ \left \langle G(x_T^v-x_T),x_T^v-x_T \right \rangle  \bigg \}
 		\nonumber\\\ge&\frac{1}{2} \mathbb{E}\bigg \{ \int_{0}^{T} \Big(\left \langle \widetilde{Q}_t(x^{v}_t-x_t), x^{v}_t-x_t\right \rangle +\left \langle \widetilde{R}_t(v _t-u_t),v _t-u_t\right\rangle\\&~~~~~~+2\left\langle\widetilde{R}_t^{-1}\widetilde{S}_t(x^{v}_t-x_t),\widetilde{R}_t(v_t-u_t)\right\rangle\Big)dt\bigg \}\nonumber\\
 		=&\frac{1}{2} \mathbb{E}\bigg \{ \int_{0}^{T} \Big(\left \langle (\widetilde{Q}_t -\widetilde{S}_t^{\top}\widetilde{R}_t^{-1}\widetilde{S}_t)(x^{v }_t-x_t), x^{v}_t-x_t\right \rangle\nonumber\\&~~~~~~+ \left\langle\widetilde{R}_t\big[(v_t-u_t)+\widetilde{R}_t^{-1}\widetilde{S}_t(x^{v}_t-x_t)\big],(v_t-u_t)+\widetilde{R}_t^{-1}\widetilde{S}_t(x^{v}_t-x_t)\right\rangle\Big)dt\bigg \}.\nonumber
 	\end{align}
By virtue of Assumption \ref{ass:4.1}-(iii),(iv), it is easy to verify that $J\big(v(\cdot )\big)-J\big(u(\cdot )\big)\ge 0$. This is to say that $u(\cdot)$ is the optimal control of Problem(LQDL).

For the uniqueness, if there exists another optimal control $\bar{u}(\cdot)$ and its corresponding state is $x^{\bar{u}}(\cdot)$. Thus, $J\big(\bar{u}(\cdot )\big)=J\big(u(\cdot )\big)$. Coming back to \eqref{eq:6.9} and combining with Assumption \ref{ass:4.1} and Lemma \ref{lem:4.1}, we derive
\begin{equation}\label{eq:4.10}
	\begin{aligned}
		0=&J\big(\bar{u}(\cdot )\big)-J\big(u(\cdot )\big)\\
		 \ge&\frac{1}{2} \mathbb{E}\bigg \{ \int_{0}^{T} \Big(\left \langle (\widetilde{Q}_t -\widetilde{S}_t^{\top}\widetilde{R}_t^{-1}\widetilde{S}_t)(x^{\bar{u} }_t-x_t), x^{\bar{u} }_t-x_t\right \rangle\\&+ \left\langle\widetilde{R}_t\big[(\bar{u}_t-u_t)+\widetilde{R}_t^{-1}\widetilde{S}_t(x^{\bar{u}}_t-x_t)\big],(\bar{u}_t-u_t)+\widetilde{R}_t^{-1}\widetilde{S}_t(x^{\bar{u}}_t-x_t)\right\rangle\Big)dt\bigg \}\\
		 \ge&\frac{1}{2} \mathbb{E}\Biggl\{\int_{0}^{T}\gamma\left\langle\widetilde{R}_t(\bar{u}_t-u_t),\bar{u}_t-u_t\right\rangle dt\Biggr\}.
	\end{aligned}
\end{equation}
Due to the nonnegative definiteness of $\widetilde{R}_t$ and the fact that $\gamma>0$, \eqref{eq:4.10} implies that $\bar{u}(\cdot)=u(\cdot)$. The proof of uniqueness is completed.    	
\end{proof}
\subsection{Backward LQ stochastic control problem}\label{sec:4.2}
In this section, the control system is given by a linear ABSDEL:
\begin{equation}\label{eq:4.11}
	\left\{
	\begin{aligned}
		dy_t=&\Big (A_ty_t+\bar{A}_{t }\mathbb{E}^{\mathcal{F}_t}[y_{t+\delta }]+B_tz_t+\bar{B}_{t }\mathbb{E}^{\mathcal{F}_t}[z_{t+\delta }]+\sum_{i=1}^{\infty } C^{(i)}_tk^{(i)}_t+\sum_{i=1}^{\infty } \bar{C}^{(i)}_{t }\mathbb{E}^{\mathcal{F}_t}[k^{(i)}_{t+\delta }]
		+D_tv_t+\bar{D}_{t}v_{t-\delta }\Big )dt\\&+z_tdW_t+\sum_{i=1}^{\infty }k_t^{(i)}dH_t^{(i)},\quad t\in[0,T], \\
		y_T=&b,\quad y_t= z_t= k_t=0,\quad t\in (T,T+\delta ],\\
		v_t=&\iota_t, \quad t\in[-\delta,0),		
	\end{aligned}\right .
\end{equation}
where $b\in L^2_{\mathcal{F}_T}(\Omega;\mathbb{R}^n)$ and $\iota_t\in C(-\delta,0;\mathbb{R}^m)$. Moreover, $A_t,\bar{A}_t,B_t,\bar{B}_t,C^{(i)}_t,\bar{C}^{(i)}_t\in L^{\infty}_{\mathbb{F}}(0,T;\mathbb{R}^{n\times n})$, $D_t\in L^{\infty}_{\mathbb{F}}(0,T;\mathbb{R}^{n\times m})$, $\bar{D}_t\in L^{\infty}_{\mathbb{G}}(0,T;\mathbb{R}^{n\times m})$, where $\bar{D}_t=0$, when $t\in[0,\delta]$. Let $\mathcal{U}_{ad}$ denote the set of stochastic processes $v(\cdot)$ with the form:
\begin{equation}
	\left\{\begin{aligned}
		&v_t=\iota_t,\quad t \in [-\delta,0),\\
		&v_t=v_t\in  L^{2} _{\mathbb{F}}(0,T;\mathbb{R}^{m} ),\quad t\in [0,T],\nonumber
	\end{aligned}\right.
\end{equation}
which is called an addmissible control.

Similar to before, we give a cost functional as follows:
\begin{equation}\label{eq:4.12}
	\begin{aligned}
		J\big(v(\cdot )\big)=\frac{1}{2} \mathbb{E}\bigg \{&\int_{0}^{T} \Big ( \left \langle Q_ty_t,y_t \right \rangle + \left \langle \bar{Q}_{t }y_{t+\delta },y_{t+\delta } \right \rangle+\left \langle L_tz_t,z_t \right \rangle + \left \langle \bar{L}_{t }z_{t+\delta },z_{t+\delta } \right \rangle
		+\sum_{i=1}^{\infty}\left \langle G^{(i)}_tk^{(i)}_t,k^{(i)}_t \right \rangle \\&+ \sum_{i=1}^{\infty}\left \langle \bar{G}^{(i)}_{t }k^{(i)}_{t+\delta },k^{(i)}_{t+\delta } \right \rangle+\left \langle R_tv_t,v_t \right \rangle
		+\left \langle \bar{R}_{t }v_{t-\delta },v_{t-\delta } \right \rangle\Big )dt
		+\left \langle My_0,y_0 \right \rangle \bigg \},
	\end{aligned}
\end{equation}
where $Q_t, \bar{Q}_t, L_t, \bar{L}_t, G^{(i)}_t, \bar{G}^{(i)}_t \in L^{\infty}_{\mathbb{F}}(0,T;\mathbb{S}^{n})$, $R_t, \bar{R}_t \in L^{\infty}_{\mathbb{F}}(0,T;\mathbb{S}^{m})$ and $M$ is an $\mathcal{F}$-measurable $n\times n$ symmetric and nonnegative definite matrix-valued bounded random variable. Moreover, $\bar{Q}_t=\bar{L}_t=\bar{G}_t=0$, when $t\in [-\delta,0)$ and $\bar{R}_t=0$, when $t\in(T,T+\delta]$. Similarly, we assume that $Q_t+\bar{Q}_{t-\delta}$, $L_t+\bar{L}_{t-\delta}$, $G_t+\bar{G}_{t-\delta}$ are all nonnegative definite and $R_t+\mathbb{E}^{\mathcal{F}_t}[\bar{R}_{t+\delta}]$ has the similar condition as the $R_t+\mathbb{E}^{\mathcal{F}_t}[\bar{R}_{t+\delta}]$ in subsection \ref{sec:4.1}.

Now, our problem in this section is proposed as follows:\\
\textbf{Problem(LQAL).} Find an admissible control $u(\cdot)\in \mathcal{U}_{ad}$ such that
\begin{equation}
	J\big(u(\cdot )\big)=\underset{v(\cdot )\in \mathcal{U}_{ad} }{\textrm{inf}}J\big(v(\cdot )\big) .
\end{equation}
Likewise, we give the Hamiltonian system of ABSDEL \eqref{eq:4.11} as follows:
\begin{equation}\label{eq:4.14}
	\left \{\begin{aligned}
		dx_t=&-\Big (A_t^{\top} x_t+\bar{A}_{t-\delta }^{\top }x_{t-\delta } +Q_t^{\top} y_t+\bar{Q}_{t-\delta }^{\top }y_{t }\Big )dt
		-\Big (B_t^{\top} x_t+\bar{B}_{t-\delta }^{\top }x_{t-\delta } +L_t^{\top} z_t+\bar{L}_{t-\delta }^{\top }z_{t }\Big )dW_t\\
		&-\Big (\sum_{i=1}^{\infty } (C_t^{(i)\top} x_{t-}+\bar{C}_{t-\delta }^{(i)\top }x_{(t-\delta)-} +G_t^{(i)\top} k^{(i)}_t+\bar{G}_{t-\delta }^{(i)\top }k^{(i)}_{t }\Big )dH_t^{(i)},\quad t\in[0,T],\\
		dy_t=&\Big (A_ty_t+\bar{A}_{t }\mathbb{E}^{\mathcal{F}_t}[y_{t+\delta }]+B_tz_t+\bar{B}_{t }\mathbb{E}^{\mathcal{F}_t}[z_{t+\delta }]+\sum_{i=1}^{\infty } C^{(i)}_tk^{(i)}_t+\sum_{i=1}^{\infty } \bar{C}^{(i)}_{t }\mathbb{E}^{\mathcal{F}_t}[k^{(i)}_{t+\delta }]
		+D_tu_t+\bar{D}_{t}u_{t-\delta }\Big )dt\\&+z_tdW_t+\sum_{i=1}^{\infty }k_t^{(i)}dH_t^{(i)},\quad t\in[0,T],\\
		x_0=&-My_0,\quad x_t=0,\quad u_t=\iota_t,\quad t\in [-\delta ,0),\\
		y_T=&b,\quad x_t=y_t=z_t=k_t=0 , \quad t\in (T,T+\delta ],\\
		(D_t ^\top &x_t+\mathbb{E}^{\mathcal{F}_t}[\bar{D}_{t+\delta } ^\top x_{t+\delta}])+(R_t+\mathbb{E}^{\mathcal{F}_t}[\bar{R}_{t+\delta }])u_t=0.
	\end{aligned}\right .
\end{equation}

Obviously,  the Hamiltonian system \eqref{eq:4.14} is also driven by a FBSDELDA which satisfies all the assumptions in section \ref{sec:3}, so by Theorem \ref{thm:3.1}, we have the following theorem.
\begin{thm}\label{thm:4.2}
	The Hamiltonian system \eqref{eq:4.14} admits a unique solution $(\theta(\cdot),u(\cdot))$. Moreover, $u(\cdot)$ is the unique optimal control of Problem(LQAL).
\end{thm}
Similar to the proof of Theorem \ref{thm:4.1}, we will still explore the difference between $J\big(v(\cdot )\big)$ and $J\big(u(\cdot )\big)$, and then apply It\^{o}'s formula to $\left \langle x_t,y^v_t-y_t \right \rangle $. Due to the similarity between the proof of the two, we omit the detailed proof of the latter.
\begin{rmk}\label{rmk:4.1}
	(i) In fact, in subsection \ref{sec:4.1}, the coefficients of the FBSDELDA in Hamiltonain system \eqref{eq:4.4} are verified to satisfy Assumption \ref{ass:3.1} and Assumption \ref{ass:3.2}-(i)-case 1, Assumption \ref{ass:3.2}-(ii), Assumption \ref{ass:3.2}-(iii), whereas Assumption \ref{ass:3.2}-(i)-case 2 corresponds to backward LQ stochastic control problem; (ii) If we replace the positive definiteness and nonnegative definiteness of the coefficients of the cost functional in \eqref{eq:4.2} and \eqref{eq:4.12} with the negative definiteness and non-positive definiteness, then the coefficients of FBSDELDAs in Hamiltonain system \eqref{eq:4.4} and \eqref{eq:4.14} satisfy Assumption \ref{ass:3.1}, Assumption \ref{ass:3.2}-(i), (ii) and the symmetrical version of Assumption \ref{ass:3.2}-(iii) in Remark \ref{rmk:3.1}. In this situation, the problem of minimize the cost functional will become maximizing it. In summary, the domination-monotonicity conditions exactly correspond to four classes of optimal control problems.
\end{rmk}

\section*{Acknowledgments }

The authors would like to thank anonymous referees for helpful comments and suggestions which
improved the original version of the paper. Q. Meng was supported by the Key Projects of Natural
Science Foundation of Zhejiang Province (No. Z22A013952) and the National Natural Science
Foundation of China ( No.12271158 and No. 11871121).  Maoning Tang was supported by the Natural Science
Foundation of Zhejiang Province (No. LY21A010001)

\end{document}